\documentclass[a4paper,twoside,11pt,reqno]{amsart}
\usepackage[margin=1in]{geometry}
\usepackage[english]{babel}
\usepackage[utf8]{inputenc}
\usepackage{amsmath}
\usepackage{amssymb}
\usepackage{amsfonts}
\usepackage{amsthm}
\usepackage{bm}
\usepackage{mathrsfs}
\usepackage{hyperref}
\usepackage{comment}
\hypersetup{
	colorlinks=true,
	linkcolor=blue,
	filecolor=blue,
	citecolor=purple,      
	urlcolor=cyan,
}
\usepackage{xcolor}
\usepackage{dsfont}
\usepackage{physics}
\usepackage{stmaryrd}
\usepackage{enumitem}

\usepackage{tikz-cd}
\usepackage{tikz}
\usepackage{tikzsymbols}
\usepackage{mathtools}
\usetikzlibrary{matrix}

\usepackage[matrix,arrow,cmtip]{xy} 
\SelectTips{cm}{}
\newdir{(}{{}*!/-5pt/@^{(}}
\newdir{(x}{{}*!/-5pt/@_{(}}
\newdir{+}{{}*!/-9pt/{}}
\newdir{>+}{@{>}*!/-9pt/{}}
\entrymodifiers={+!!<0pt,\fontdimen22\textfont2>}

\theoremstyle{plain}
\newtheorem{theorem}[subsubsection]{Theorem}

\newtheorem{lemma}[subsubsection]{Lemma}
\newtheorem{proposition}[subsubsection]{Proposition}
\newtheorem{notation}[subsubsection]{Notation}
\newtheorem{corollary}[subsubsection]{Corollary}
\theoremstyle{definition}
\newtheorem{definition}[subsubsection]{Definition}
\newtheorem{construction}[subsubsection]{Construction}

\newtheorem{example}[subsubsection]{Example}
\theoremstyle{remark}
\newtheorem{remark}[subsubsection]{Remark}

\numberwithin{equation}{section}

\newcommand{\CC}{\mathbb{C}}
\newcommand{\RR}{\mathbb{R}}
\newcommand{\QQ}{\mathbb{Q}}
\newcommand{\ZZ}{\mathbb{Z}}

\newcommand{\HH}{\mathrm{H}}

\newcommand{\Pic}{\mathrm{Pic}}

\newcommand{\Br}{\mathrm{Br}}

\newcommand{\Hom}{\mathrm{Hom}}

\newcommand{\Coh}{\mathrm{Coh}}

\newcommand{\cA}{\mathcal{A}}

\newcommand{\cL}{\mathcal{L}}

\newcommand{\dX}{\mathcal{X}}
\newcommand{\cB}{\mathcal{B}}

\newcommand{\sX}{\mathscr{X}}

\newcommand{\too}{\longrightarrow}

\newcommand{\tX}{\widetilde{X}}
\newcommand{\cEnd}{\mathcal{E}nd}

\begin{document}
\title{Morita theory of twisted sheaves on $\mu_{n}$-gerbes of line bundles}
		\date{}
\author{Yeqin Liu and Yu Shen}
  \address{Department of Mathematics, University of Michigan, 530 Church St,
  Ann Arbor, MI 48109, USA}
\email{yqnl@umich.edu}
\address{Department of Mathematics, Michigan State University, 619 Red Cedar Road, East Lansing, MI 48824, USA}
  \email{shenyu5@msu.edu}	
  
\subjclass[2020]{14F22}

\begin{abstract}
    We study Morita theory of twisted sheaves on $\mu_{n}$-gerbes of line bundles $\sX$. In this context,  we find explicit equivalent conditions for when two Azumaya algebras on $\sX$ are Morita equivalent. Additionally, we   provide an example showing that   Căldăraru's Conjecture does not hold for Deligne--Mumford stacks in general.
    
    \end{abstract}
  
\maketitle

\section{Introduction}
In \cite{Gab62}, Gabriel showed that a Noetherian scheme $X$ can be reconstructed from $\operatorname{Coh}(X)$, the abelian  category of coherent sheaves on $X$:

\begin{theorem}[{\cite{Gab62, Ros04}}]\label{reconsctruction theorem} Let $X$ and $Y$ be Noetherian schemes.  Then
\begin{center}
   
  $\operatorname{Coh}(X)\cong \operatorname{Coh}(Y)$ as abelian categories  $\Longleftrightarrow$ $X\cong Y$.
  \end{center}

\end{theorem}

This theorem can be generalized to the setting of twisted sheaves. A notable special case is the following conjecture by Căldăraru \cite{Cal00}, 
which was proved in various generalities
by \cite{CS07,Per09,  Ant16, CG13} for quasi-compact quasi-separated algebraic spaces.


\begin{theorem}[Conjectured by \cite{Cal00}, proved by {\cite{CS07,Per09,  Ant16, CG13}}]\label{big theorem}
Let $X$  be a Noetherian algebraic space over a field $k$, and $\alpha, \beta \in \HH_{\acute{e}t}^{2}(X, \mathbb{G}_{m})$. Then
\begin{center}
$\operatorname{Coh}(X,\alpha)\cong \operatorname{Coh}(X,\beta) $ as $k$-linear abelian categories $  \Longleftrightarrow$ 
there exists a $k$-automorphism $f: X\to X$, such that $f^{*}\beta=\alpha$.
\end{center}
\end{theorem}

Note that Theorem \ref{reconsctruction theorem} does not hold for Deligne--Mumford stacks. 
\begin{example}\label{trivial example}
    $\operatorname{Coh}(\mathbf{B}\mu_{2, \mathbb{C}})\cong \operatorname{Coh}(\operatorname{Spec}(\mathbb{C}\times \mathbb{C}))$ as $\mathbb{C}$-linear categories, but $\mathbf{B}\mu_{2,\mathbb{C}} \not \cong \operatorname{Spec}(\mathbb{C}\times \mathbb{C})$ as stacks.
\end{example}
However, it was previously unknown whether Theorem \ref{big theorem} extends to Deligne--Mumford stacks.  In this paper, we study Morita theory of twisted sheaves on $\mu_{n}$-gerbes of line bundles over smooth projective varieties.  We get a complete characterization of Morita equivalent Azumaya algebras. As a corollary, we show that Căldăraru's Conjecture \textbf{does not} hold for Deligne--Mumford stacks in general (see Corollary \ref{caldararu}), which was not known before.

To state the main result, we need the following settings. Let $\dX$ be a Noetherian Deligne--Mumford stack over a field $k$. We say that two Azumaya algebras $\cA$ and $\cB$ on $\dX$ are \emph{Morita equivalent} if $\Coh(\dX, \cA)\cong \Coh(\dX,\cB)$ as $k$-linear categories. Let $n\in \ZZ_{>0}$, and  assume the base field $k$ contains $n$-th roots of unity with $\operatorname{char}{k}=p \nmid n$ ($p$ can be 0). Let $X/k$ be a smooth projective variety, and $\sX$ be a $\mu_{n}$-th  gerbe of line bundle  over $X$ (see Definition \ref{rootgerbe}). For example, $\sX$ can be $\mathbf{B}\mu_{n,X}$. Then there is an isomorphism (see Proposition \ref{exact sequence}):
\begin{equation}\label{intropsi}
    \psi: \HH^{1}_{\acute{e}t}(X,\mu_{n})\oplus \Br(X) \xrightarrow{\sim} \Br(\sX),\quad  (q: \tX\to X, [\cA']) \mapsto [\cA].
\end{equation}
Here, $q: \widetilde{X}\to X$ is  a $\mu_{n}$-torsor corresponding to an element  in $\HH^{1}_{\acute{e}t}(X,\mu_{n})$. Let $\cA$ and $\cB$ be Azumaya algebras over $\sX$ such that $[\cA]=\psi(q_{1}: \widetilde{X}_{1}\to X, [\cA'])$ and $[\cB]=\psi(q_{2}: \widetilde{X}_{2}\to X, [\cB'])$ under (\ref{intropsi}).
Our main result is a full classification of Morita equivalent Azumaya algebras over a $\mu_{n}$-gerbe of line bundle.

\begin{theorem}[{Theorem \ref{Theorem root gerbe general}}]\label{main theorem 2} Let $\cA$ and $\cB$ be two Azumaya algebras on $\sX$. Then $\cA$ and $\cB$ are Morita equivalent if and only if there is an isomorphism $f: \widetilde{X}_{1}
\xrightarrow{\sim}
\widetilde{X}_{2}$ as algebraic varieties over $k$ (not as $\mu_{n}$-torsors), and $[q_{1}^{*}\cA']=[f^{*}q_{2}^{*}\cB']$ in $\Br(\widetilde{X}_{1})$.
 \end{theorem}

 As an application, we show that Căldăraru's Conjecture does not hold for stacks in general.

\begin{corollary}[{Corollary \ref{caldararu}}]
Căldăraru's conjecture
does not hold for Deligne--Mumford stacks.   There exist  two Azumaya algebras $\cA$ and $\cB$ on $\mathbf{B}\mu_{2,\mathbb{R}}$ which are Morita equivalent, but there does not exist an automorphism $\varphi: \mathbf{B}\mu_{2,\mathbb{R}}\to \mathbf{B}\mu_{2,\mathbb{R}} $ over  $\mathbb{R}$ such that $[\cA]=[\varphi^{*}\cB]$ in $\Br(\mathbf{B}\mu_{2,\mathbb{R}})$.
    
\end{corollary}

We first prove Theorem \ref{main theorem 2} in the case where $X=\operatorname{Spec} k$ and $\Br(k)=0$. In this setting, we obtain a useful and simple criterion for determining when two Azumaya algebras are Morita equivalent. We also show that Căldăraru's Conjecture holds in this special case.

\begin{theorem}[Theorem {\ref{Theorem for field}}, Corollary \ref{hold calda}]\label{main theorem}
Assume $\Br(k)=0$. Let  $\cA, \cB$ be two Azumaya algebras on $\mathbf{B}\mu_{n,k}$.  Then $\cA$ and $\cB$ are Morita equivalent if and only $[\cA]$ and $[\cB]$ generate the same subgroup in $\Br(\mathbf{B}\mu_{n,k})$. 

In this case (assuming $\Br(k)=0$), Căldăraru's Conjecture holds. To be precise, two Azumaya algebras $\cA$ and $\cB$ on  $\mathbf{B}\mu_{n,k}$ are Morita equivalent if and only if there exists an automorphism $\varphi: \mathbf{B}\mu_{n,k} \to \mathbf{B}\mu_{n,k} $ such that $[\cA]=[\varphi^{*}\cB]$ in $\Br(\mathbf{B}\mu_{n,k})$.
\end{theorem}

Our main result (Theorem \ref{Theorem root gerbe general}) also exhibits the following interesting phenomenon: 
 \begin{center}
 \textit{
 A decomposable category can become indecomposable after a Brauer twist.}
 \end{center}
 This is discussed in Example \ref{elliptic}. Finally,  note that although  this paper  focuses on  $\mu_{n}$-gerbes of line bundles over smooth projective varieties, it is reasonable to expect that Theorem \ref{main theorem 2} holds for $\mu_{n}$-gerbes of line bundles over general Noetherian schemes using the same approach.

\subsection{Outline of Paper} In Section \ref{section 2}, we review the definitions and basic properties of the Brauer group and twisted sheaves on Deligne--Mumford stacks. We also introduce the notion of  $\mu_{n}$-gerbes of line bundles $\sX$, which are the main class of stacks considered in this paper. Finally, we review the properties of sheaves of noncommutative algebras. In Section \ref{section Brauer}, we calculate the Brauer group $\Br(\sX)$ and provide the detailed description of its structure. In Section \ref{section 4}, we prove Theorem \ref{main theorem 2}  for $X=\operatorname{Spec} k$. We will also provide an example showing that   Căldăraru's Conjecture does not hold for stacks in general, see Example \ref{example mortia} and Corollary \ref{caldararu}.
In Section \ref{section 5}, we globalize the arguments to prove Theorems \ref{main theorem} in full generality. 

\subsection{Acknowledgment}

We would like to thank Izzet Coskun, Rajesh Kulkarni, Alexander Perry, and Shitan Xu for many useful discussions.

The second author was partially supported by NSF grant DMS-2101761.
\subsection{Notation} Fix a positive integer $n$. In this paper, we assume that the field $k$ contains all $n$-th roots of unity with $\operatorname{char}(k)=p\nmid n$ ($p$ can be 0). The classifying stack of the $n$-th cyclic group $\mu_{n}$ over $X$ is denoted by $\mathbf{B}\mu_{n,X}$.
All cohomology groups in this paper are understood to be taken in the étale topology.
  
\section{Preliminaries}\label{section 2}

\subsection{Brauer group on Deligne-Mumford stack} In this subsection, we collect basic facts about the Brauer group. For more details about Brauer group in general, see \cite{Gro68,Shi19,AM20}. 
\begin{definition}
   An \textit{Azumaya algebra} over a Deligne--Mumford stack $\dX$ is a sheaf of quasi-coherent $\mathcal{O}_{\dX}$-algebra $\cA$ such that $\cA$ is \'etale locally on $\dX$ isomorphic to $M_{m}(\mathcal{O}_{\dX})$, the sheaf of $m\times m$ matrices over $\mathcal{O}_{\dX}$, for some $m\geq 1$.
\end{definition}
\begin{example}
   \begin{enumerate}
   \item If $E$ is a vector bundle on $\dX$ of rank $m> 0$, then $\mathcal{E}nd(E)$ is an Azumaya algebra on $\dX$.
   \item The \emph{quaternion algebra} 
   $$\mathbb{H}=\{a+bi+cj+dij:~ a,b,c,d\in \mathbb{R} \},\quad \text{where } i^2=j^2=-1, ij=-ji$$
   is an Azumaya algebra over $\mathbb{R}$.
   \end{enumerate}
\end{example}

If $\cA$ and $\mathcal{B}$ are Azumaya algebras on $\mathcal{X}$, then $\cA\otimes_{\mathcal{O}_{\dX}}\mathcal{B}$ is an Azumaya algebra. We give the following definitions of Brauer groups.

\begin{definition}
    Two Azumaya algebras $\cA$ and $\mathcal{B}$ are \textit{Brauer equivalent} if there are vector bundles $E$ and $F$ such that $\cA\otimes_{O_{\dX}}\mathcal{E}nd(E)\cong \mathcal{B}\otimes_{\mathcal{O}_{\dX}}\mathcal{E}nd(F)$. The \textit{Brauer group} $\Br(\dX)$ of  $\dX$ is the set of isomorphism classes of Azumaya algebras modulo Brauer equivalence, where $[\cA]+[\cB]=[\cA\otimes_{\mathcal{O}_{\dX}}\cB],$ and $-[\cA]=[\cA^{\operatorname{opp}}]$. Here  $\cA^{\operatorname{opp}}$ is the opposite algebra of $\cA$.
\end{definition}

\begin{definition}
    Let $\dX$ be a quasi-compact and quasi-separated Deligne--Mumford stack. The \textit{cohomological Brauer group} of $\dX$ is defined to be $\Br'(\dX):=\HH^{2}(\dX, \mathbb{G}_{m})_{\operatorname{tors}}$, the torsion subgroup of $\HH^{2}(\dX, \mathbb{G}_{m})$.
\end{definition}

Note that there exists a natural injective map $\Br(\dX)\hookrightarrow \Br'(\dX)$.
This map is often an isomorphism
if $\dX$ admits nice properties, as shown in the following proposition.
However in general, this map may not be surjective \cite{CTS21}.

\begin{proposition}[{\cite[Corollary 2.1.5]{Shi19}}]\label{two Brauer are same}
If $\dX$ is a smooth separated generically  tame Deligne--Mumford stack  over $k$ with quasi-projective coarse moduli space, then we have $$ \Br'(\dX)=\HH^{2}(\dX, \mathbb{G}_{m})=\Br(\dX). $$
\end{proposition}

The following are several examples of Brauer groups.

\begin{example}\label{trivial Brauer group}
\begin{enumerate} 
    \item If $C$ is a smooth curve over $\bar{k}$, then $\Br(C)=0$. Let $K(C)$ be the function field of $C$. Then $\Br(K(C))=0$.
    \item(\cite[Theorem 5.13]{CTS21}) $\Br(k)\cong \Br(\mathbb{P}^{m}_{k})$ for any field $k $.
    \item (\cite[Corollary 5.2.6]{CTS21}) If $X$ is a smooth projective stably rational variety over $\bar{k}$, then $\Br(X)=0$.
    \end{enumerate}    
\end{example}

\subsection{Twisted sheaf and Morita theory} In this subsection, we recall the notion of twisted sheaves. For more details, see \cite{Cal00, lieblich2007moduli}.

\begin{definition}[\cite{Cal00}]\label{twisted sheaves}
    Let $\dX$ be a Noetherian Deligne--Mumford stack over a field $k$. Let $\alpha\in \check{C}^{2}(\dX, \mathcal{O}_{\dX}^{*})$ be a Čech 2-cocycle (in the  \'etale topology), given by  an open cover $\mathcal{U}=\{U_{i}\}_{i\in I}$ and sections $\alpha_{ijk}\in \Gamma(U_{i}\cap U_{j} \cap U_{k}, \mathcal{O}_{\dX}^{*})$.  An $\alpha$-\emph{twisted sheaf} on $\dX$ consists of a collection $(\{ \mathcal{F}_{i}\}_{i\in I}, \{ \varphi_{ij}\}_{i,j\in I})$ with $\mathcal{F}_{i}$ being a sheaf of $\mathcal{O}_{\dX}$-modules on $U_{i}$ and $\varphi_{ij}: \mathcal{F}_{j}|_{U_{i}\cap U_{j}}\to\mathcal{F}_{i}|_{U_{i}\cap U_{j}}$ being isomorphisms such that
\begin{enumerate}
    \item $\varphi_{ii}$ is the identity for all $i\in I$;
    \item $\varphi_{ij}=\varphi_{ji}^{-1}$;
    \item $\varphi_{ij}\circ \varphi_{jk}\circ \varphi_{ki}$ is multiplication by $\alpha_{ijk}$ on $\mathcal{F}_{i}|_{U_{i}\cap U_{j}\cap U_{k} }$ for all $i,j,k \in I.$
\end{enumerate}
A \textit{homomorphism} $f$ between $\alpha$-twisted sheaves $\mathcal{F}, \mathcal{G}$ consists of a collection of maps $f_{i}: \mathcal{F}_{i} \to \mathcal{G}_{i} $ for each $i\in I$ such that $f_{i}\circ \varphi_{\mathcal{F}, ij}=\varphi_{\mathcal{G},ij}\circ f_{j} $ for all $i,j \in I$.
\end{definition}

\begin{remark}
In \cite[Definition 2.1.2.2]{lieblich2007moduli}, Lieblich gives a more general definition of twisted sheaves, which coincides with Definition \ref{twisted sheaves} in our setting by \cite[Proposition 2.1.3.3]{lieblich2007moduli}.
\end{remark}

\begin{definition}[\cite{Cal00}]
    An $\alpha$-sheaf $\mathcal{F}$ on $\dX$ is called \emph{(quasi-)coherent} if all the underlying sheaves $\mathcal{F}_{i}$ are (quasi-)coherent. The category of $\alpha$-twisted sheaves, $\alpha$-twisted quasicoherent sheaves, and $\alpha$-twisted coherent sheaves are denoted by $\operatorname{Mod}(\dX, \alpha), \operatorname{QCoh}(\dX,\alpha)$, and $ \Coh(\dX,\alpha)$, respectively.
\end{definition}
It turns out that the category of $\alpha$-twisted sheaves depend only on its cohomology class.
\begin{lemma}[{\cite[Lemma 1.2.8]{Cal00}}]\label{same cohomology class}
    If $\alpha$ and $\alpha'$ represent the same element of $\check{\HH}^{2}(\dX, \mathcal{O}_{\dX}^{*})$, then $\operatorname{Mod}(\dX, \alpha)=\operatorname{Mod}(\dX,\alpha')$. 
\end{lemma}

Let $\cA$ be an Azumaya algebra. The category of right $\cA$-modules, quasi-coherent $\cA$-modules, and coherent right $\cA$-Modules are denoted by $\operatorname{Mod}(\dX, \cA), \operatorname{QCoh}(\dX,\cA)$, and $ \Coh(\dX,\cA)$, respectively. Let $\alpha:=[\cA]$ in $\Br(\dX)$. Then we have the following proposition.

\begin{proposition}[{\cite[Theorem 1.3.7]{Cal00}}]\label{equivalence between twisted sheaf and Azumaya algebra}
    There is a $k$-linear functor $F:\operatorname{Mod}(\dX, \alpha )\to \operatorname{Mod}(\dX,\cA)$ which is an equivalence. Furthermore, this will induce equivalences of $k$-linear categories $F|_{\operatorname{QCoh(\dX,\alpha)}}: \operatorname{QCoh}(\dX,\alpha)\to \operatorname{QCoh}(\dX,\cA)$ and $F|_{\operatorname{Coh(\dX,\alpha)}}: \operatorname{Coh}(\dX,\alpha)\to \operatorname{Coh}(\dX,\cA)$.
\end{proposition}

In this paper, by Proposition \ref{equivalence between twisted sheaf and Azumaya algebra}, we will consider twisted sheaves as right modules over Azumaya algebras.


Now, we will introduce the notion of Morita equivalence. First, we briefly recall the classical definition. For more details, see \cite{Lam12}. Let $R$ be a ring with identity. Let  $\operatorname{Mod}_{R}$ be the categories  right $R$-modules.

\begin{definition}\label{Morita ring}
    Let $R, T$ be two rings. $R$ is \textit{Morita equivalent} to $T$ if $\operatorname{Mod}_{R}$ is equivalent to $\operatorname{Mod}_{T}$ as abelian categories.
\end{definition}
The following definition introduces the notion of \emph{progenerators.}

\begin{definition}\label{progenerator}
    Let $R$ be a ring. A right $R$-module $E$ is said to be an $R$-\textit{progenerator} if it satisfies the following two conditions:
    \begin{enumerate} 
        \item E is finitely generated projective;
        \item E is a generator, i.e. the functor $\Hom_{R}(E, - )$ from $\operatorname{Mod}_{R}$ to the category of abelian groups is faithful.
    \end{enumerate}
\end{definition}
In fact, over a commutative ring $R$, progenerators are equivalent to vector bundles on $\operatorname{Spec}(R)$.
\begin{lemma}[{\cite[18.11 and Ex. 2.24]{Lam12}}]\label{commutative}
If $R$ is a commutative ring. Then $E$ is $R$-progenerator if and only if $E$ is a finitely generated projective $R$-module with positive rank on each component of $\operatorname{Spec}(R)$.
\end{lemma}

We have the following \emph{Fundamental Theorem of Morita Theory}. 
\begin{theorem}[{Fundamental Theorem of Morita Theory}]\label{fundamental theory}
Let $R, T$ be rings. Then $R$ is Morita equivalent to $T$ if and only if there exists an $R$-progenerator $E$ such that $T\cong \operatorname{End}_{R}(E)$. In this case, the functors
$$ -\otimes_{T}E: {\operatorname{Mod}_{T}}\to \operatorname{Mod}_{R}  \quad  \operatorname{and} \quad  -\otimes_{R}E^{\vee}: {\operatorname{Mod}_{R}}\to {\operatorname{Mod}_{T}} $$
are mutually inverse.    
\end{theorem}

We can generalize the definition of Morita equivalence to sheaves of algebras on stacks in a natural way. Let $\dX$ be a  Noetherian Deligne-Mumford stack over $k$ and $\cA$ be a sheaf of coherent $\mathcal{O}_{\dX}$ algebra on $\dX$. Let $\Coh(\dX, \cA)$ be the category of coherent right $\cA$-modules.

\begin{definition}\label{Morita stack}
Let $\cA$ and $\cB$ be sheaves of coherent $\mathcal{O}_{\dX}$ algebras on $\dX$.   We say $\cA$ is \textit{Morita equivalent} to $\cB$ if  $\Coh(\dX, \cA)$ is equivalent to $\Coh(\dX, \cB)$ as $k$-linear abelian categories.
\end{definition}

The following proposition reveals the relation between Brauer group and Morita theory.
\begin{proposition}[{\cite[Theorem 1.3.15]{Cal00}}]\label{Azu--Br}
     Let $A, B$ be two Azumaya algebras over $k$. Then $A$ is Morita equivalent to $B$ if and only if $[A]=[B]$ in the Brauer group $\Br(k)$.
\end{proposition}
By Proposition \ref{same cohomology class}, Brauer equivalence always implies Morita equivalence. On the other hand, in general, Morita equivalence does not imply Brauer equivalence, as we will see later.

\subsection{$\mu_{n}$-gerbe of line bundle} In this subsection, we recall the definition and properties of $\mu_{n}$-gerbes of line bundles.

\begin{definition}[{\cite[Example 3.9.12]{Alp23}}]\label{rootgerbe}
Fix $n\in \ZZ_{>0}$.
Let $X$ be a scheme and $L$ be a line bundle on $X$, which has the classifying morphism $[L]: X\to \mathbf{B}\mathbb{G}_{m}$. Let $n:\mathbf{B}\mathbb{G}_{m}\to \mathbf{B}\mathbb{G}_{m}$ be the morphism induced from the $n$-th power map $\mathbb{G}_{m}\to \mathbb{G}_{m}: t\to t^n$. The \textit{$\mu_n$-gerbe of a line bundle $L$}, denoted by \( \mathscr{X} \), is defined as the fiber product
\[ \begin{tikzcd}
\mathscr{X} \arrow{r} \arrow[swap]{d} {p}&\mathbf{B}\mathbb{G}_{m} \arrow{d}{n} \\%
X \arrow{r}{[L]}& \mathbf{B}\mathbb{G}_{m}.
\end{tikzcd}
\]
\end{definition}

\begin{proposition}\label{rootgerbeDM}
The $\mu_{n}$-gerbe $\sX$ of line bundle L  in Definition \ref{rootgerbe} has following properties:
\begin{enumerate}
    \item $p: \sX\to X$ is the coarse moduli space.
    \item $\sX$ is a Deligne--Mumford stack.
    \item If $X=\operatorname{Spec}(A)$ is an affine scheme, and $L=\mathcal{O}_{X}$ is trivial in the construction, then $$\sX \cong [\operatorname{Spec}(A)/\mu_{n}]\cong  \mathbf{B}\mu_{n,\operatorname{Spec}(A)}, $$
where $\mu_{n}$ acts trivially on $X$.
\end{enumerate}

\begin{remark}
    Note that we have the short exact sequence on $X$:
    \begin{equation}\label{mu exact sequence}
        1\to \mu_{n} \to \mathbb{G}_{m}\xrightarrow {n}\mathbb{G}_{m} \to 1.
    \end{equation}
Taking cohomology groups, we get a map $ \iota: \Pic(X)\to \HH^{2}(X,\mu_{n})$. The $\mu_{n}$-gerbe $\sX$ of line bundle $L$ in Definition \ref{rootgerbe} is corresponds to the $\mu_{n}$-gerbe $\iota(L)$.
    
    \end{remark}

\end{proposition}    

By \cite[Section 5]{IU15}, there exists the universal object $(\mathcal{M},\Phi)$ on $\sX$, where $\mathcal{M}$ is a line bundle on $\sX$ and $\Phi:\mathcal{M}^{\otimes n}\to p^{*}L$ is an isomorphism of line bundles.

\begin{definition}
    Let $\mathcal{M}$ be the universal line bundle on $\sX$ and $i\in \mathbb{Z}$. We will use $\rho_{i}$ to denote the $i$-th power of $\mathcal{M}$.
    $$\rho_{i}:=\mathcal{M}^{\otimes i}.$$
\end{definition}

\begin{lemma}[{\cite[Theorem 1.5]{IU15}}]\label{decomposition}
  The category $\Coh(\sX)$ splits as the following direct sum:
$$\operatorname{Coh}(\mathscr{X})\cong \operatorname{Coh}(X)\rho_{0} \oplus \operatorname{Coh}(X)\rho_{1}\oplus\cdots \oplus \operatorname{Coh}(X)\rho_{n-1},$$
where $p^{*}: \Coh(X) \to \Coh(\sX) $ is the fully faithful embedding.
\end{lemma}
Note that the decomposition for $\Coh(\sX)$ is orthogonal, which also induces an orthogonal decomposition for $D^{b}(\sX)$.

\subsection{Sheaf of finite algebra} \label{section sheaf}  In this subsection, we review basic properties of sheaves of noncommutative algebras over varieties. For more details, see \cite{Kuz06,Kuz08}.

Let $X$ be a smooth proper variety over $k$ and $\cB_{X}$ be a sheaf of $\mathcal{O}_{X}$-algebra which is locally free of finite rank as $\mathcal{O}_{X}$-module. 
Let $\operatorname{QCoh}(X, \cB_{X})$ be the category of quasicoherent sheaves of right $\cB_{X}$-modules. Note that this category has enough injective and enough locally free objects. We will consider the pair $(X, \cB_{X})$ as a noncommutative algebraic variety. 
\begin{definition}

Let $(X, \cB_{X}), (Y, \cB_{Y})$ be such two pairs. A morphism $\widetilde{f}: (X, \cB_{X})\to (Y, \cB_{Y})$ is a pair $(f, f_{\cB})$, where $f: X\to Y$ is a morphism of algebraic varieties and $f_{\cB}: f^{*}\cB_{Y}\to \cB_{X}$ is a morphism of $f^{*}\mathcal{O}_{Y}\cong \mathcal{O}_{X}$-algebras.
    \end{definition}
As the usual cases, we can define the pushforward and pullback functors of the morphism $\widetilde{f}$.  Let $\Coh(X, \cB_{X})$ be the category of coherent sheaves of right $\cB_{X}$-modules, and let $D^{b}(X, \cB_{X}):=D^{b}(\Coh(X, \cB_{X}))$ be its bounded derived category.
\begin{definition}
    Let $\widetilde{f}: (X, \cB_{X})\to (Y, \cB_{Y})$. We can associate the pushforward $\widetilde{f}_{*}:\operatorname{Coh}(X, \cB_{X}) \to\operatorname{Coh}(Y, \cB_{Y})$ and the pullback $\widetilde{f}^{*}:\operatorname{Coh}(Y, \cB_{Y})\to\operatorname{Coh}(X, \cB_{X})$ as follows:
$$\widetilde{f}_{*}(F):=f_{*}F, \quad \widetilde{f}^{*}(G)=f^{*}G\otimes_{f^{*}\cB_{Y}}\cB_{X}. $$
Then $\widetilde{f}_{*}$ is left exact and $\widetilde{f}^{*}$ is right exact, and there are derived functors:
$$ R\widetilde{f}_{*}: D^{b}(\Coh(X, \cB_{X}))\to D^{b}(\Coh(Y, \cB_{Y})), \quad L\widetilde{f}^{*}:  D^{b}(\Coh(Y, \cB_{Y}))\to D^{b}(\Coh(X, \cB_{X})) .$$
The functors $\widetilde{f}_{*}, \widetilde{f}^{*}, R\widetilde{f}_{*}, L\widetilde{f}^{*}$, etc., behave similarly to the usual functors between varieties. All  propositions for usual functors still hold. The following propositions will be needed in this paper.

\begin{lemma}[{\cite[Lemma D.17]{Kuz06}}]
The functor $L\widetilde{f}^{*}$ is left adjoint to $R\widetilde{f}_{*}$.
\end{lemma}

For simplicity, we will use $\widetilde{f}_{*}, \widetilde{f}^{*}$, and $\otimes$ to represent the derived functors $R\widetilde{f}_{*}, L\widetilde{f}^{*}$, and $\otimes^{L}$.
We will need the projection formula for sheaves on noncommutative varieties. 

\begin{lemma}[{\cite[Lemma D.12]{Kuz06}}]\label{projection formula}
Let $\widetilde{f}:(X, \cB_{X})\to (Y, \cB_{Y})$ be a morphism. Suppose $F\in D^{b}(X, \cB_{X}^{opp}), G\in D^{b}(Y, \cB_{Y})$, then we have
$$ \widetilde{f}_{*}({\widetilde{f}^{*}G}\otimes_{\cB_{X}}F)\cong G \otimes_{\cB_{Y}}\widetilde{f}_{*}F.   $$
\end{lemma}
    
\end{definition}

\begin{lemma}[{\cite[Lemma D.4]{Kuz06}}]\label{flat}
    A sheaf $F\in \Coh(X, \cB_{X})$ is locally projective over $\cB_{X}$ in the Zariski topolopgy if and only if $F$ is locally free as an $\mathcal{O}_{X}$-module. 
\end{lemma}
So if there is a map $\widetilde{f}:(X, \cB_{X})\to (Y, \cB_{Y})$, then $\cB_{X}$ is a locally projective $f^{*}\cB_{Y}$-bimodule.

\begin{corollary}\label{Kuz flat}
    Let $\widetilde{f}:(X, \cB_{X})\to (Y, \cB_{Y})$ be a morphism. If $f^{*}: \Coh(Y)\to \Coh(X)$ is exact, then the functor $\widetilde{f}^{*}:\operatorname{Coh}(Y, \cB_{Y})\to\operatorname{Coh}(X, \cB_{X})$ is exact.
\end{corollary}
\begin{proof}
Let $0\to F\to G \to H \to 0 $ be an exact sequence   $\Coh(Y, \cB_{Y})$. Since $f^{*}$ is flat, by Lemma \ref{flat}, we have the following short exact sequence $$0\to  f^{*}F\otimes_{f^{*}\cB_{Y}} \cB_{X}\to f^{*}G\otimes_{f^{*}\cB_{Y}} \cB_{X} \to  f^{*}H\otimes_{f^{*}\cB_{Y}} \cB_{X} \to 0 . $$
Thus, $\widetilde{f}^{*}$ is exact.
\end{proof}

We also need the base change formula. Let $\widetilde{f}=(f,\operatorname{id}):(X, f^{*}\cB_{S} )\to (S, \cB_{S})$ and $\widetilde{g}: (Y,\cB_{Y})\to (S,\cB_{S}) $ be two morphisms. Let $p: X\times_{S}Y\to X$ and $q: X\times_{S}Y \to Y $ denote the projections.
\begin{lemma}[{\cite[Lemma D.37]{Kuz06}}]\label{fiber product}
We have the following fiber product diagram:
\begin{center}
    \begin{tikzcd}
(X\times_{S}Y, p^{*}\cB_{Y} ) \arrow[r, "\widetilde{p}"] \arrow[d, "\widetilde{q}"] & (Y,\cB_{Y}) \arrow[d, "\widetilde{g}"] \\
(X,f^{*}\cB_{S}) \arrow[r, "\widetilde{f}"]                & (S,\cB_{S}),    \end{tikzcd}
\end{center}
where $\widetilde{p}=(p,\operatorname{id})$ and $\widetilde{q}=(q, q^{*}f^{*}\cB_{S}=p^{*}g^{*}\cB_{S}\to p^{*}\cB_{Y})$.
\end{lemma}

\begin{lemma}[{\cite[Lemma 2.22]{Kuz06}}]\label{flat equal}
    The natural morphism of functors $\widetilde{q}_{*}\widetilde{f}^{*}\to\widetilde{f}^{*}\widetilde{g}_{*} $ is an isomorphism if and only if $q_{*}f^{*}\to f^{*}g_{*} $ is an isomorphism.
\end{lemma}

\begin{remark} The constructions and lemmas above also hold for smooth proper Deligne--Mumford stacks over $k$.
\end{remark}

\section{Brauer group of $\mu_{n}$-gerbe of line bundle}\label{section Brauer}
In this section, we compute the Brauer group $\Br(\sX)$ of the $\mu_{n}$-gerbe of line bunlde $L$ and provide an explicit description of the elements in  $\Br(\mathbf{B}\mu_{n,k})$. The main result of this section is Proposition \ref{general non trivial}.

 We will fix $n\in \ZZ_{>0}$ in this section. Recall that  $k$ is a field of $\operatorname{char}(k)=p$ with $p\nmid n$ and contains $n$-th roots of unity. Let $X$ be a smooth projective variety over a field $k$. We begin by reviewing the computation of the Brauer group $\Br(\mathbf{B}\mu_{n,X})$, where $\mathbf{B}\mu_{n,X}$ denotes the classifying stack of the $n$-th cyclic group $\mu_n$ over $X$.

\begin{lemma}[{\cite[Proposition 3.2]{AM20}}]\label{Brauer group of classifying stack}

 Let $p: \mathbf{B}\mu_{n,X} \to X$ denote the coarse moduli space. 
Then we have:
\[
R^{0}p_{*}\mathbb{G}_m = \mathbb{G}_m, \quad R^{1}p_{*}\mathbb{G}_m = \mu_n, \quad R^{2}p_{*}\mathbb{G}_m = 0.
\]
Moreover, the Leray spectral sequence $$ E_{2}^{p_{1},q_{1}}= \HH^{p_{1}}(\operatorname{Spec}X, R^{q_{1}}p_{*}\mathbb{G}_{m} )\Longrightarrow \HH^{p_{1}+q_{1}}(\mathbf{B}\mu_{n,X}, \mathbb{G}_{m}) $$
yields a split short exact sequence:
\[
0 \to \Br(X) \xrightarrow{p^*} \Br(\mathbf{B}\mu_{n,X}) \xrightarrow{q} \mathrm{H}^1(X, \mu_n) \to 0.
\]
\end{lemma}

Now let $L$ be a line bundle on $X$, and let
$p: \sX \to X$
be the  $\mu_{n}$-gerbe of line bundle $L$, as in Definition~\ref{rootgerbe}. To generalize Lemma~\ref{Brauer group of classifying stack} to compute $\Br(\sX)$, we need the following lemmas.

\begin{lemma}\label{injective of pushforward}
The pullback map
\[
p^{*} : \Br(X) \to \Br(\sX)
\]
is injective and admits a splitting.
\end{lemma}
\begin{proof}
There is a Zariski affine open cover $\{U_{i}\}$ of $X$, such that $L|_{U_{i}}\cong \mathcal{O}_{U_{i}}$ for each $U_{i}$. Let $U_{ij}:=U_{i}\cap U_{j}$. For each $U_{i}$, we have the following maps 
 $$U_{i}\xrightarrow{\pi_{i}} \mathbf{B}\mu_{n,U_{i}}\xrightarrow{p_{i}} U_{i},  $$
 where $p_{i}\circ \pi_{i}=\operatorname{id}$. Since $X$ is smooth, by  \cite[Theorem 3.2.2]{CTS21}, we have the following commutative diagram
 \begin{center}
 \begin{tikzcd}
0 \arrow[r] & \Br(X) \arrow[r] \arrow[d, "p^{*}"] & \bigoplus_{i}\Br(U_{i}) \arrow[r] \arrow[d, "\bigoplus_{i} p^{*}_{i}"] & \bigoplus_{i,j}\Br(U_{ij}) \arrow[d, "\bigoplus_{i,j}p^{*}_{ij}"] \\
0 \arrow[r] & \Br(\sX)\arrow[r] \arrow[d, "\pi^{*}",  dashed]               & \bigoplus_{i}\Br(\mathbf{B}\mu_{n,U_{i}}) \arrow[r] \arrow[d, "\bigoplus_{i}\pi^{*}_{i}"]               & \bigoplus_{i,j}\Br(\mathbf{B}\mu_{n,U_{ij}}) \arrow[d, "\bigoplus_{i,j}\pi^{*}_{ij}"] \\
0 \arrow[r] & \Br(X) \arrow[r]  & \bigoplus_{i}\Br(U_{i})\arrow[r]                & \bigoplus_{i,j}\Br(U_{ij}).
\end{tikzcd} 
\end{center}
Since \( p_{i}^{*} \) are injective, \( p^{*} \) is injective. On the other hand, by diagram chasing, we know there exists a morphism \( \pi^{*}: \Br(\sX) \to \Br(X) \). Since \( \pi_{i}^{*} \circ p_{i}^{*} = \operatorname{id} \), we have \( \pi^{*} \circ p^{*} = \operatorname{id} \). We complete the proof.
\end{proof}
\begin{lemma}\label{pushforward}
We have
$R^{0}p_{*}\mathbb{G}_{m}=\mathbb{G}_{m}, R^{1}p_{*}\mathbb{G}_{m}=\mu_{n}, R^{2}p_{*}\mathbb{G}_{m}=0. $
\end{lemma}

\begin{proof}
There is an affine open cover $\{U_{i}\}$ of $X$, such that $L|_{U_{i}}\cong \mathcal{O}_{U_{i}}$ for each $U_{i}$. For each $U_{i}$,   we have the following Cartesian diagram 
\begin{center}
\begin{tikzcd}
\mathbf{B}\mu_{n,U_{i}} \arrow[r] \arrow[d, "p_{i}"] & \sX \arrow[d, "p"] \\
U_{i} \arrow[r]                & X               
\end{tikzcd}
\end{center}
 By Lemma \ref{Brauer group of classifying stack}, we have 
$R^{0}p_{i*}\mathbb{G}_{m}=\mathbb{G}_{m}, R^{1}p_{i*}\mathbb{G}_{m}=\mu_{n}, R^{2}p_{i*}\mathbb{G}_{m}=0.$ Thus, we get the lemma.
\end{proof}

Let $K(X)$ be the function field of $X$, and let $U$ be an open subvariety of $X$. Let $\eta: \operatorname{Spec}K(X) \hookrightarrow U \xhookrightarrow{\eta_{U}} X$ denote the inclusion morphisms. We have the following lemma.

\begin{lemma}\label{injective of H}
The morphisms $\eta^{*}: \HH^{1}(X, \mu_{n}) \to \HH^{1}(\operatorname{Spec}K(X), \mu_{n})$ and $\eta^{*}_{U}: \HH^{1}(X, \mu_{n}) \to \HH^{1}(U, \mu_{n})$ are injective.
\end{lemma}

\begin{proof}
    We have the spectral sequence $$E_{2}^{p,q}=\HH^{p}(X, R^{q}\eta_{*}\mu_{n})\Longrightarrow \HH^{p+q}(\operatorname{Spec} K(X), \mu_{n}). $$
    So we have an injective morphism $$\HH^{1}(X, \eta_{*}\mu_{n})\hookrightarrow \HH^{1}(\operatorname{Spec} K(X), \mu_{n}).$$
    Note that we have the natural map $\mu_{n}\to \eta_{*}\mu_{n}.$ We claim it is an isomorphism. Indeed, let $U\to X$ be an \'etale morphism. Since $X$ is smooth, $X$ is normal. Hence, $U$ is normal. If it is connected, then it is integral. This shows that the map $\mu_{n} \to \eta_{*}\mu_{n}$ is an isomorphim. Thus, the morphism $\eta^{*}$ is injective. Since $\eta^{*}$ factors through $\eta_{U}^{*}$, $\eta_{U}^{*}$ is injective.
\end{proof}

By Lemmas \ref{injective of pushforward} and  \ref{pushforward},   the Leray spectral sequence
$$ E_{2}^{p_{1},q_{1}}= \HH^{p_{1}}(\operatorname{Spec}X, R^{q_{1}}p_{*}\mathbb{G}_{m} )\Longrightarrow \HH^{p_{1}+q_{1}}(\sX, \mathbb{G}_{m}) $$ will induce a short exact sequence
$$0\to \HH^{2}( X, \mathbb{G}_{m})\xrightarrow{p^{*}}  \HH^{2}(\sX, \mathbb{G}_{m})\xrightarrow{q} \HH^{1}( X, \mu_{n}).$$

\begin{lemma}\label{HH map}

There exists a morphism $i: \HH^{1}(X,\mu_{n})\to \HH^{2}(\sX,\mathbb{G}_{m})$ such that $q\circ i=\operatorname{id}.$ 
    
\end{lemma}

\begin{proof}

Using the notation in Lemma \ref{injective of pushforward}, and by Lemmas \ref{Brauer group of classifying stack} and \ref{injective of H}, we obtain the following commutative diagram:

\begin{center}
\begin{tikzcd}
0 \arrow[r] & \HH^{1}(X,\mu_{n}) \arrow[r] \arrow[d, dashed, "i"]              & \oplus_{i}\HH^{1}(U_{i},\mu_{n}) \arrow[r] \arrow[d] & \oplus_{i,j}\HH^{1}(U_{ij},\mu_{n}) \arrow[d] \\
0 \arrow[r] & \Br(\sX) \arrow[r] \arrow[d, "q"] & \oplus_{i}\Br(\mathbf{B}\mu_{n,U_{i}}) \arrow[r] \arrow[d, "q_{i}"] & \oplus_{i,j}\Br(\mathbf{B}\mu_{n,U_{ij}}) \arrow[d, "q_{ij}"] \\
0 \arrow[r] & \HH^{1}(X,\mu_{n}) \arrow[r]                        & \oplus_{i}\HH^{1}(U_{i},\mu_{n}) \arrow[r]                & \oplus_{i,j}\HH^{1}(U_{ij},\mu_{n})               
\end{tikzcd}
\end{center}

By diagram chasing, we know there exists an injective morphism $i: \HH^{1}(X, \mu_{n}) \to \Br(\sX)$ such that $q \circ i = \operatorname{id}$.

\end{proof}

Now, from Lemma \ref{HH map}, we can get the following proposition.

\begin{proposition}\label{exact sequence}
   We have the short exact sequence
   \begin{equation}\label{exact sequence of Brauer group}
        0\to \Br(X)\xrightarrow{p^{*}} \Br(\sX)\xrightarrow{q} \HH^{1}(X, \mu_{n})\to 0.
   \end{equation}
    The short exact sequence \ref{exact sequence of Brauer group} is split. 
\end{proposition}

Let $X=\operatorname{Spec} k$ and $\sX=\mathbf{B}\mu_{n,k}$. Recall that we have the maps $\operatorname{Spec}k\xrightarrow{\pi} \mathbf{B}\mu_{n,k}\xrightarrow{p} \operatorname{Spec} k$, where $p\circ \pi=\operatorname{id}$.
     We will provide an explicit description of elements in $\Br(\mathbf{B}\mu_{n,k})$ in terms of matrices.

\subsection{Explicit matrix description of $\Br(\mathbf{B}\mu_{n,k})$.}
In this subsection, we describe $\Br(\mathbf{B}\mu_{n,k})$ explicitly. The idea was first used in \cite{{lieblich2011period}}. We first deal with fields $k$ with $\Br(k)=0$.

\begin{proposition}\label{Brauer group of  field}
    Let $k$ be a field with $\Br(k)=0$. Then there is an isomorphism $$\psi: \Br(\mathbf{B}\mu_{n,k})\xrightarrow{\sim} k^{*}/k^{*n}. $$
\end{proposition}

\begin{proof}
    Since $\Br(k)=0$, by Proposition \ref{exact sequence}, $\Br(\mathbf{B}\mu_{n,k})\cong \HH^{1}(\operatorname{Spec} k, \mu_{n})\cong k^{*}/k^{*n}$
\end{proof}
To describe the isomorphism $\psi$ in Proposition \ref{Brauer group of field} explicitly, we first establish the following lemma.
\begin{lemma}[{\cite[Corollary 2.4.2]{GS17}}]\label{SN theorem}
  Let $m\in \mathbb{Z}_{>0}. $ We have the following short exact sequence
    $$1\too k^{*} \too \operatorname{GL}_{m}(k) \too \operatorname{Aut}(M_{m}(k)) \too 1,$$
where the map $\operatorname{GL}_{m}(k)\to \operatorname{Aut}(M_{m}(k))$ is given by $: B \mapsto (M \mapsto B^{-1}MB). $
\end{lemma}
Let $\cA$ be an Azumaya algebra of degree $m$ over $\mathbf{B}\mu_{n,k}$. Assume $\Br(k)=0$. Then  $\pi^{*}\cA\simeq M_{m}(k)$ as $\mu_{n}$-equivariant algebras. Let $\zeta$ be a generator of $\mu_{n}$. By Lemma \ref{SN theorem}, the action of $\mu_{n}$ on $M_{m}(k)$ is given by
  $$\zeta \cdot M =B^{-1}MB \quad \operatorname{for} \quad M\in M_{m}(k),  $$
for some  $B\in \operatorname{GL}_{m}(k)$ such that  $B^{n}$ is a scalar matrix.  Conversely, given a matrix $B$ such that $B^{n}$ is a scalar matrix, we can get an Azumaya algebra on $\textbf{B}\mu_{n,k}$. 
\begin{notation}\label{matrix notation}
Let $B\in \operatorname{GL}_{m}(k)$ so that $B^{n}$ is a scalar matrix.    We will use $M_{m,B}(k)$ to denote the $\mu_{n}$-equivariant algebra $M_{m}(k)$, where the action  is given by $\zeta\cdot M=B^{-1}MB$. 
\end{notation}
In the remainder of this subsection, we assume that the morphisms between two $\mu_{n}$-equivariant algebras are $\mu_{n}$-equivariant.
By Lemma \ref{SN theorem}, we have the following lemma.
\begin{lemma}\label{isomorphism class of matrix}
Let $M_{m,B}(k), M_{m,B'}(k)$ be the two $\mu_{n}$-equivariant $k$-algebra. Then $M_{m,B}(k)\cong M_{m,B'}(k)$ if and only if there is $C\in \operatorname{GL}_{m}(k)$, such that $[B]=[C^{-1}B'C]$ in $\operatorname{PGL}_{m}(k)$. 
\end{lemma}

\begin{definition}
    Let $B, B'\in GL_{m}(k)$. We say $B\sim B'$ if there exist $C\in \operatorname{GL}_{m}(k) $ such that $[B]=[C^{-1}B'C]$ in $  \operatorname{PGL}_{m}(k).$
\end{definition}
By Lemma \ref{isomorphism class of matrix}, we get the following proposition.
 \begin{proposition}\label{equivalence of Azumaya}
  Assume $\Br(k)=0$. There is a one-to-one correspondence 
\[
 \left\{\begin{tabular}{l}
\textit{isomorphism classes of degree}  $m$ \ \\ \textit{Azumaya algebras on}  $\mathbf{B}\mu_{n,k}$
\end{tabular}
\right\} \longleftrightarrow \biggl\{ B\in \operatorname{GL}_{m}(k)\bigg\vert B^{n} \textit{is a scalar matrix} \biggr\}\bigg/\sim. 
\]
\end{proposition}
Let $\cA$ be an Azumaya algebra of degree $m$ on $\mathbf{B}\mu_{n,k}$ with trivial Brauer class. Then there is a $\mu_{n}$-equivariant vector space $V$ of rank $m$ such that $\cA\cong \operatorname{End}(V)$. 

On the one hand, since $V$ is $\mu_{n}$-equivariant, it induces a group homomorphism $\rho: \mu_{n}\to \operatorname{GL}_{m}(k)$. This produces the matrix $B=\rho(\zeta)\in \operatorname{GL}_{m}(k)$ where $B^{n}=\operatorname{Id}$. 


On the other hand, the $\mu_{n}$-action on $\operatorname{End}(V)=\Hom_{k}(V, V)$ is induced by the action of $\mu_{n}$ on $V$. Since $\zeta \cdot v = Bv$ for $v \in V$,
 the induced action on $\operatorname{End}(V)$ is given by conjugation: for any $M \in M_m(k) = \operatorname{Hom}_k(V, V)$, we have $\zeta \cdot M = B^{-1} M B$.
Thus, the trivial Azumaya algebra $\cA$ on $\mathbf{B}\mu_{n,k}$ corresponds to $M_{m,B}$ for some matrix $B \in \operatorname{GL}_m(k)$ satisfying $B^n = \operatorname{Id}$. So we get the following correspondence.
\begin{proposition}\label{trivial azumaya}
   Assume $\Br(k)=0$. There is a one-to-one correspondence 
\[
 \left\{\begin{tabular}{l}
\textit{isomorphism classes of degree}  $m$ \textit{Azumaya} \ \\ \textit{ algebras on}  $\mathbf{B}\mu_{n,k}$ \textit{with trivial Brauer class}
\end{tabular}
\right\} \longleftrightarrow \biggl\{ B\in \operatorname{GL}_{m}(k)\bigg\vert B^{n}=\operatorname{Id} \biggr\}\bigg/\sim. 
\]
\end{proposition}

By Proposition \ref{equivalence of Azumaya}, we also explicitly describe the tensor product of Azumaya algebras.
\begin{lemma}\label{basic lemma}
   Let $\cA, \cA'$ be Azumaya algebras on $\mathbf{B}\mu_{n,k}$. Suppose $\cA\cong M_{m,B}(k)$ and $\cA'\cong M_{m',B'}(k)$, then  $ \cA\otimes \cA'\cong M_{mm', B\otimes B' }(k)$, where $B\otimes B'$ is the Kronecker product of $B$ and $B'$.
       
\end{lemma}

\begin{definition}
    Let $S:=\{B\in \bigcup_{m\geq 1} \operatorname{GL}_{m}(k): ~B^{n}\ \text{is a scalar matrix}  \}$. We define a map $$\gamma: S\to k^{*}, \quad B\mapsto \gamma(B), $$
where $\gamma(B)=a$ if $B^{n}=a\operatorname{Id}$.
\end{definition}

Now we explicitly describe the isomorphism $\psi$ in Definition \ref{Brauer group of  field}.
\begin{proposition}\label{important propersition}
     Assume $\Br(k)=0$. Using Notation \ref{matrix notation}, the isomorphism $\psi$ is given by $$\psi: \Br(\mathbf{B}\mu_{n,k})\xrightarrow{\sim} k^{*}/k^{*n}, \quad [M_{m,B}(k)] \mapsto [\gamma(B)]. $$
\end{proposition}
\begin{proof}

If $[M_{m,B}(k)]=[M_{m',B'}(k)]$ in $\Br(\mathbf{B}\mu_{n,k})$, then there exist $\mu_{n}$-equivariant vector spaces $V, V'$ such that $M_{m,B}\otimes \operatorname{End}(V)\cong  M_{m',B'}\otimes \operatorname{End}(V')$. Suppose $\operatorname{rank}(V)=r$ and $\operatorname{rank}(V')=r'$. By Proposition \ref{trivial azumaya}, $\operatorname{End(V)}\cong M_{r,C}(k)$ and $\operatorname{End}(V')\cong M_{r',C'}(k)$ for some matrices $C, C'$ such that $C^{n}=\operatorname{Id}$ and $C'^{n}=\operatorname{Id}$. By Lemma \ref{basic lemma}, we have $M_{mr, B\otimes C}(k)\cong M_{m'r', B'\otimes C'}(k). $ By Lemma \ref{isomorphism class of matrix}, we have $$[\gamma(B)]=[\gamma(B\otimes C)]=[\gamma(B'\otimes C')]=[\gamma(B')]. $$
So $\psi$ is well defined. Since $\gamma(B\otimes B')=\gamma(B)\gamma(B')$, by  Lemma \ref{basic lemma}, $\psi$ is a group homomorphism.

Now, suppose $\psi([M_{m,B}(k)])=[1]$. Then $\gamma(B)=a^{n}$ for some $a\in k^{*}.$ By Proposition \ref{equivalence of Azumaya}, $M_{m,B}(k)\cong M_{m,a^{-1}B}(k)$. Since $(a^{-1}B)^{n}=\operatorname{Id}$, we have $[M_{m,B}(k)]=[M_{m,a^{-1}B}(k)]=0$ in the $\Br(\mathbf{B}\mu_{n,k})$. So $\psi$ is injective.

Let $a\in k^{*}$. Consider the $n\times n$ matrix
\begin{equation}\label{matrix B}
B=\begin{pmatrix}
0 & ... & 0 & 0 & a\\
1 & ... & 0 & 0 & 0 \\
... & ... & ... & ... & ... \\
0 & ... & 1 & 0 & 0 \\
0 & ... & 0 & 1 & 0
\end{pmatrix}.
\end{equation}
Since $B^{n}=a\operatorname{Id}$, we have $\psi([M_{n,B}(k)])=[a]$. So $\psi$ is surjective. 

\end{proof}

 In general, we have an isomorphism $\psi: k^{*}/k^{*n}\oplus \Br(k) \xrightarrow{\sim} \Br(\mathbf{B}\mu_{n,k})$. The following proposition describes $\psi$ explicitly.

\begin{proposition}\label{general non trivial}
Let $a\in k^{*}$ and the matrix $B=$(\ref{matrix B}). Using Notation \ref{matrix notation}, the isomorphism $\psi$ is given by
$$\psi: k^{*}/k^{*n}\oplus \Br(k) \xrightarrow{\sim} \Br(\mathbf{B}\mu_{n,k}): ([a], [\cA])\mapsto [M_{n,B}(k)\otimes p^{*}\cA], $$
where $\cA$ is an Azumaya algebra over $k$.

\end{proposition}

\section{Categories of coherent modules over Azumaya algebras}\label{section 4}
 In this section, we study the categories of twisted sheaves on $\mu_{n}$-gerbes of line bundles. We will first prove Theorem \ref{main theorem}, and then  get Theorem \ref{main theorem 2} in the case $X=\operatorname{Spec} k$. As a corollary, we  provide an example  showing   that Căldăraru's Conjecture does not hold for Deligne--Mumford stacks.
 
\begin{lemma}\label{exact functor} Recall that  we have the map $p: \sX\to X.$ The functors $p_{*}$ and $p^{*}$ are exact. Moreover, $p_{*}(\mathcal{O}_{\sX})=\mathcal{O}_{X}$
\end{lemma}
\begin{proof}
    By \cite[Lemma 2.3.4]{AV02}, $p_{*}$ is exact and $p_{*}(\mathcal{O}_{\sX})=\mathcal{O}_{X}$.  Locally,  we  have the following commutative diagram:
    \begin{center}
\begin{tikzcd}
U_{i} \arrow[r, "\pi"] \arrow[rd, "\operatorname{id}"] & \lbrack  U_{i}/\mu_{n} \rbrack \arrow[d, "p  "] \\
                                 &  U_{i}.           
\end{tikzcd}
   \end{center} 
   Since $\pi$ is a $\mu_{n}$-Galois cover and $p\circ\pi=\operatorname{id}$, $p^{*}$ is exact.
\end{proof}

\begin{lemma}\label{identity lemma}
 Let $F\in \Coh(\sX)$, and suppose $F=F_{0}\rho_{0}\oplus...\oplus F_{n-1}\rho_{n-1}$. Then we have   $\operatorname{id}\xrightarrow{\sim}p_{*}p^{*}$ and $p_{*}F\cong F_{0}. $
\end{lemma}
\begin{proof}
    This directly follows from Lemma \ref{decomposition}.
\end{proof}
Let $\cA$ be an Azumaya algebra on $\sX$. There is a decomposition for $\cA$ as an $\mathcal{O}_{\dX}$-module: $$\cA=\cA_{0}\rho_{0}\oplus ...\oplus \cA_{n-1}\rho_{n-1},$$ where $\cA_{i}=p_{*}(\cA\otimes \rho_{-i})$. Since $\cA$ is locally free as an $\mathcal{O}_{\sX}$-module, $\cA_{i}$ will  be locally free as  $\mathcal{O}_{X}$-modules for all $i$.  Using the notions introduced in the subsection \ref{section sheaf}, we can define a map between the noncommutative algebraic varieties $(\sX, \cA)$ and $(X, p_{*}\cA)=(X,\cA_{0})$.

\begin{definition}
    There is a map $\widetilde{p}=(p, p_{\cA} ): (\sX, \cA )\to (X, p_{*}\cA), $ where $p_{\cA}:p^{*}p_{*}\cA \to \cA$ is the adjunction map. 
\end{definition}
Recall that $\Coh(\sX, \cA)$ is the category of coherent right $\cA$-modules. As in Lemma \ref{exact functor}, we can show that the functors $\widetilde{p}_{*}$ and $\widetilde{p}^{*}$ are also exact. 

\begin{lemma}
    The functors $\widetilde{p}_{*}: \Coh(\sX, \cA)\to \Coh(X,  p_{*}\cA) $ and $\widetilde{p}^{*}: \Coh(X, p_{*}\cA)\to \Coh(\sX,\cA)$ are exact.
\end{lemma}
\begin{proof}
    Since $p_{*}$ is exact, $\widetilde{p}_{*}$ is exact. Since $p^{*}$ is exact, by Lemma \ref{Kuz flat}, $\widetilde{p}^{*}$ is exact.
\end{proof}
Thus, the funcotrs $\widetilde{p}_{*}$ and $\widetilde{p}^{*}$ are the  same as the derived functors $R\widetilde{p}_{*}$ and $L\widetilde{p}^{*}$. As in Lemma \ref{identity lemma}, we have the following lemma.

\begin{lemma}
    For any $F\in \Coh(X, p_{*}\cA)$, we  have $F\cong \widetilde{p}_{*}\widetilde{p}^{*}F$ in $\Coh(\sX,\cA)$.
\end{lemma}
\begin{proof}
The map $\widetilde{p}: (\sX,\cA) \to (X, p_{*}\cA) $ admits the following decomposition:
\begin{center}

\begin{tikzcd}
(\sX, \cA) \arrow[r, "\widetilde{p^{e}}"] & (\sX, p^{*}p_{*}\cA ) \arrow[r, "\widetilde{p^{s}}"] & (X, p_{*}\cA ),
\end{tikzcd}    
\end{center}
where $\widetilde{p^{e}}=(\operatorname{id}, p_{\cA})$ and $\widetilde{p^{s}}=(p, \operatorname{id} ).$ Since $\widetilde{p^{s}}_{*}$ and $\widetilde{p^{s}}^{*}$ are also exact, by projection formula \ref{projection formula},  we have 
$$\widetilde{p}_{*}\widetilde{p}^{*}F\cong p_{*}(p^{*}F\otimes_{p^{*}p_{*}\cA}\cA)\cong \widetilde{p^{s}}_{*}(\widetilde{p^{s}}^{*}F\otimes_{p^{*}p_{*}\cA}\cA)\cong F \otimes_{p_{*}\cA}\widetilde{p^{s}}_{*}\cA\cong F.   $$
    
\end{proof}
By Lemma \ref{identity lemma}, we can show that the functor $\widetilde{p}^{*}$ is fully faithful. This leads to the following proposition.
\begin{proposition}\label{fully faithful}
    The functor $\widetilde{p}^{*}: \Coh(X, p_{*}\cA)\to \Coh(\sX, \cA)$ is  fully faithful.
\end{proposition}

\begin{proof}
    Let $F, G\in \Coh(X,p_{*}\cA)$. Then by Lemma \ref{identity lemma}, we have
$$\Hom_{\cA}(\widetilde{p}^{*}F, \widetilde{p}^{*}G)=\Hom_{p_{*}\cA}(F, \widetilde{p}_{*}\widetilde{p}^{*}G)=\Hom_{p_{*}\cA}(F, G).$$    
\end{proof}
 It turns out that for some Azumaya algebras $\cA$, the functor $\widetilde{p}^{*}$ is an equivalence, as  indicated by the following lemma.
\begin{lemma}\label{equivalence lemma}
Let $F=F_{0}\rho_{0}\oplus\cdots F_{n-1}\rho_{n-1}$ be a vector bundle on $\sX$. Let $\cA:=\mathcal{E}nd(F)$.
If $F_{i}\not =0$ for all $i$, then   $\widetilde{p}^{*}:\Coh(X, p_{*}\cA)\to \Coh(\sX, \cA)$ is  an  equivalence.
\end{lemma}
\begin{proof}
By Lemma \ref{identity lemma}, we have $\operatorname{id}\xrightarrow{\sim}\widetilde{p}_{*}\widetilde{p}^{*}$. In order to show $\widetilde{p}^{*}$ is an equivalent functor, it is enough to prove $\widetilde{p}^{*}\widetilde{p}_{*}\xrightarrow{\sim} \operatorname{id}$. Let $H\in \Coh(\sX,\cA).$ Then
we have a distinguished triangle in $D^{b}(\sX, \cA)$ : $$ \widetilde{p}^{*}\widetilde{p}_{*}H \to H \to G^{\bullet} \to \widetilde{p}^{*}\widetilde{p}_{*}H[1], $$
where $G^{\bullet}$ is the cone of $\widetilde{p}^{*}\widetilde{p}_{*}H \to H$.
Apply  $\widetilde{p}_{*}$ to the distinguished triangle, we get $$\widetilde{p}_{*}\widetilde{p}^{*}\widetilde{p}_{*}H\to \widetilde{p}_{*}H \to \widetilde{p}_{*}G^{\bullet} \to\widetilde{p}_{*}\widetilde{p}^{*}\widetilde{p}_{*}H^{\bullet}[1].$$
Since the first arrow is an isomorphism, we have $\widetilde{p}_{*}G^{\bullet}=0.$

Let $F^{\vee}:=\mathcal{H}om(F, \mathcal{O}_{\sX})=F'_{0}\rho_{0}\oplus \cdots F_{n-1}'\rho_{n-1}$. Then, $F'_{0}\not = 0$ for all $i.$ By Theorem \ref{fundamental theory}, there is a Morita equivalent functor  $$ -\otimes_{\mathcal{O}_{X}} F^{\vee} : \Coh(\sX) \to \Coh(\sX, \cA). $$ Thus, there exists a complex  $G'^{\bullet}\in D^{b}(\sX)$ such that $G^{\bullet}\cong G'^{\bullet}\otimes_{\mathcal{O}_{X}}F^{\vee} $ as $\cA$-module complexes. Note that an $\cA$-module can be realized as an $\mathcal{O}_{\sX}$-module. So $G^{\bullet}\cong G'^{\bullet}\otimes_{\mathcal{O}_{X}} F^{\vee}$ as $\mathcal{O}_{\sX}$-module complexes.

Suppose $G^{\bullet}\not =0$, then $G'^{\bullet}\not  =0$. Note that we have an orthogonal decomposition: $$D^{b}(\sX)=D^{b}(X)\rho_{0}\oplus...\oplus D^{b}(X)\rho_{n-1}.$$
Assume $G^{\bullet}=G^{\bullet}_{0}\rho_{0}\oplus ... \oplus G^{\bullet}_{n-1}\rho_{n-1}$ and $G'^{\bullet}=G'^{\bullet}_{0}\rho_{0}\oplus ... \oplus G'^{\bullet}_{n-1}\rho_{n-1}$. Then $ \exists i$, $ 0\leq  i\leq n-1$, such that $G'^{\bullet}_{i}\not =0$.
\begin{enumerate}
    \item If $G'^{\bullet}_{0}\not = 0$. Since $F'_{0}\not = 0$ by assumption, we have  $0\not = F'_{0}\rho_{0}\otimes G'^{\bullet}_{0}\rho_{0}\subseteq G^{\bullet}_{0}\rho_{0}. $
    \item If $G'^{\bullet}_{i}\not =0$ for $0< i \leq n-1$. By  definition of the universal line on $\sX$, we have  $\rho_{n-i}\otimes\rho_{i}=L\rho_{0}$. Since $F'_{n-i}\not =0$, we have    
    $0\not = (F'_{n-i}\rho_{n-i}\otimes G'^{\bullet}_{i}\rho_{i})\otimes(L^{-1}\rho_{0})\subseteq G^{\bullet}_{0}\rho_{0}$
\end{enumerate}
So  $G^{\bullet}_{0}\rho_{0}\not =0$. Hence $\widetilde{p}_{*}G^{\bullet}=G^{\bullet}_{0}\not =0$, which contradicts  $\widetilde{p}_{*}G^{\bullet}=0$.  Thus, $G^{\bullet}=0$. So we have $\widetilde{p}^{*}\widetilde{p}_{*}H \xrightarrow{\sim} H$ and complete the proof.
\end{proof}

Now we begin to  prove the main results for $X=\operatorname{Spec} k$ and $\sX=\mathbf{B}\mu_{n,k}$. 
To do so, we first establish several lemmas. In the following, let $B$ be the $n\times n$ matrix described in \ref{matrix B}.

\begin{lemma}\label{eigenvector}
    Let $a\in k^{*}$ and $k_{1}:=k(\sqrt[n]{a})$. 
Then over the field $k_{1}$, the eigenvalues of $B$ are $\sqrt[n]{a},...,\sqrt[n]{a}\zeta^{n-1}$, where  $\zeta$ is the generator of $\mu_{n}$.
\end{lemma}
\begin{proof}
    Let $v_{i}:=((\sqrt[n]{a}\zeta^{i})^{n-1}, \cdots  \sqrt[n]{a}\zeta^{i},1 )^{\intercal}$ for $0\leq i \leq n-1$. Then $Bv_{i}=(\sqrt[n]{a}\zeta^{i})v_{i}$. Thus, $\sqrt[n]{a}\zeta^{i}$ are eigenvalues for $0\leq i \leq n-1$.
\end{proof}

Recall that $M_{n,B}(k)$ is the Azumaya algebra over $\mathbf{B}\mu_{n,k}$ associated with the matrix $B$.
\begin{lemma}\label{pushforward of algebra}
 $p_{*}M_{n,B}(k)\cong k[x]/(x^{n}-a) $ as $k$-algebras.   
\end{lemma}
\begin{proof}
    We know $p_{*}(M_{n,B}(k))= M_{n}(k)^{\mu_{n}}$, the fixed subalgebra of $M_{n}(k)$ under the action of $\mu_{n}$. Note that the action of $\mu_{k}$ on $M_{n}(k)$ is given by $\zeta \cdot M=B^{-1}MB$. Thus, $M\in p_{*}M_{n,B}(k)$ if and only $BM=MB$. By calculation, $M$ needs to be the following form
$$M=\begin{pmatrix}
a_{1} & aa_{n} & aa_{n-1} & ... & aa_{2}\\
a_{2} & a_{1} & aa_{n} & ... & aa_{3} \\
... & ... & ... & ... & ... \\
a_{n-1} & a_{n-2} & a_{n-3} & ... & aa_{n} \\
a_{n} & a_{n-1} & a_{n-2} & ... & a_{1}
\end{pmatrix}=a_{1}\operatorname{Id}+a_{2}B+...+a_{n}B^{n-1}.
$$
Let $f: k[x]/(x^{n}-a)\to p_{*}M_{n,B}(k): x\to B$. Then $f$ is an isomorphism.
  \end{proof}

\begin{lemma} \label{importan corollary} 
    Let $a, b\in k$. Then $k[x]/(x^{n}-a)$ and $k[x]/(x^{n}-b)$ are isomorphic as $k$-algebras if and only if $k(\sqrt[n]{a})=k(\sqrt[n]{b})$.
\end{lemma}
\begin{proof}
    Since $x^{n}-a$ has no multiple roots, it can be factored as $x^{n}-a=p_{1}...p_{l}$, where $p_{1}$,..,$p_{l}$ are irreducible polynomials in $k[x]$, and each pair of them is coprime. By Chinese remainder theorem, $k[x]/(x^{n}-a)\cong  k[x]/p_{1} \times ...k[x]/p_{l}$. Since all roots of $x^{n}-a$ are $\sqrt[n]{a},\sqrt[n]{a}\zeta,...,\sqrt[n]{a}\zeta^{n-1}$, $p_{i}$ are minimal polynomials of $\sqrt[n]{a}\zeta^{a_{i}}$. So $k[x]/p_{i}\cong k(\sqrt[n]{a})$ for $0\leq i\leq l$. Thus, we have $$k[x]/(x^{n}-a)\cong k(\sqrt[n]{a})\times...\times k(\sqrt[n]{a}) \cong    k(\sqrt[n]{a})^{l},$$
    where $l=n/[k(\sqrt[n]{a}):k]$. Then the statement follows.
\end{proof}

\begin{proposition}[{\cite[Theorem 8.2]{Lan02}}]\label{Kummer theory}
Let $a,b\in k$. Then $k(\sqrt[n]{a})= k(\sqrt[n]{b})$ as fields over $k$ if and only if $[a]$ and $[b]$ generate the same subgroup in $k^{*}/k^{*n}$.

\end{proposition}
Now, we can show the first main theorem in this section as follows.

\begin{theorem}\label{Theorem for field}
Assume $\Br(k)=0$. Let $\cA, \cB$ be two Azumaya algebras on  $\mathbf{B}\mu_{n,k}$. Then $\cA$ is Morita equivalent to $\cB$ if and only $[\cA]$ and $[\cB]$ generate the same subgroup in $\Br(\mathbf{B}\mu_{n,k})$.
\end{theorem}

\begin{proof}
Suppose $\psi([\cA])=[a]$, where $\psi$ is the isomorphism defined in Proposition \ref{important propersition} and $a\in k^{*}$. Let $B$ be the $n\times n$-matrix described in \ref{matrix B}.
Since $B^{n}=a\operatorname{Id}$, by Proposition \ref{important propersition}, we have  $[\cA]=[M_{n,B}(k)]$ in $\Br(\mathbf{B}\mu_{n,k})$. Let $k_{1}:=k(\sqrt[n]{a})$. Then $k_{1}$ is a Galois extension of $k$. 
By Lemma \ref{eigenvector}, over $k_{1}$, the eigenvalues the matrix $a^{-\frac{1}{n}}B$ are  $1, \zeta,...,\zeta^{n-1}$. So over $k_{1}$, $a^{-\frac{1}{n}}B$ is similar to the matrix $B_{1}$, $$B_{1}=\begin{pmatrix}
1 & 0 & 0 & ... & 0\\
0 & \zeta & 0 & ... & 0 \\
... & ... & ... & ... & ... \\
0 & 0 & ... & \zeta^{n-2} & 0  \\
0 & 0 & 0 & ... & \zeta^{n-1}
\end{pmatrix}.$$
By Lemma \ref{isomorphism class of matrix}, we have $$M_{n,B}(k_{1})\cong M_{n, a^{-1/n}B}(k_{1})\cong M_{n,B_{1}}(k_{1}).$$

On the other hand, consider the vector bundle $\Xi:=\rho_{0}\oplus ... \oplus \rho_{n-1}$ on $\mathbf{B}\mu_{n,k_{1}}$. $\Xi$ is a $\mu_{n}$-equivariant vector space over $k_{1}$. It induces a group homomorphism $\rho: \mu_{n} \to \operatorname{GL}_{n}(k_{1})$, where $\rho(\zeta)=B_{1}$. Thus, by Proposition \ref{trivial azumaya}, we have $$M_{n,B}(k_{1})\cong M_{n,B_{1}}(k_{1})\cong \mathcal{E}nd(\Xi).$$ By Lemma \ref{fiber product} and \ref{flat equal}, we have the  Cartesian diagram
\begin{center}
\begin{tikzcd}
(\mathbf{B}\mu_{n,k_{1}},M_{n,B}(k_{1})) \arrow[r,"\widetilde{q_{1}}"] \arrow[d,"\widetilde{p_{1}}"] & (\mathbf{B}\mu_{n,k}, M_{n,B}(k)) \arrow[d,"\widetilde{p}"] \\
(\operatorname{Spec}k_{1}, p_{1*}M_{n,B}(k_{1}) ) \arrow[r,"\widetilde{q}"]           & (\operatorname{Spec}k, p_{*}M_{n,B}(k)).         
\end{tikzcd}    
\end{center}
Moreover $\widetilde{p_{1}}_{*}\widetilde{q_{1}}^{*}\xrightarrow{\sim}\widetilde{q}^{*}\widetilde{p}_{*}$. By Lemma \ref{equivalence lemma}, we have $\widetilde{p_{1}}^{*}\widetilde{p_{1}}_{*}\xrightarrow{\sim} \operatorname{id}   $ and $\widetilde{p_{1}}^{*}$ defines an equivalence:
$$\widetilde{p_{1}}^{*}: \Coh( \operatorname{Spec}k_{1}, p_{1*}M_{n,B}(k_{1}))\xrightarrow{\sim} \Coh(\mathbf{B}\mu_{n,k_{1}}, M_{n,B}(k_{1})).  $$
By Proposition \ref{fully faithful}, we have a fully faithful functor $$\widetilde{p}^{*}: \Coh( \operatorname{Spec}k, p_{*}M_{n,B}(k))\to \Coh(\mathbf{B}\mu_{n,k}, M_{n,B}(k)).$$
Let $H\in \Coh(\mathbf{B}\mu_{n,k}, M_{n,B}(k))$. Then we have a distinguished triangle in $D^{b}(\mathbf{B}\mu_{n,k}, M_{n,B}(k))$:
$$ \widetilde{p}^{*}\widetilde{p}_{*} H \to H \to G^{\bullet} \to\widetilde{p}^{*}\widetilde{p}_{*}H[1].$$
Applying $\widetilde{q_{1}}^{*}$ to it, we get the following short exact sequence: 
$$ \widetilde{q_{1}}^{*}\widetilde{p}^{*}\widetilde{p}_{*} H \to \widetilde{q_{1}}^{*}H \to \widetilde{q_{1}}^{*}G^{\bullet} \to\widetilde{p}^{*}\widetilde{p}_{*}H[1].$$
Since $\widetilde{q_{1}}^{*}\widetilde{p}^{*}\widetilde{p}_{*} H\cong \widetilde{p_{1}}^{*}\widetilde{q}^{*}\widetilde{p}_{*} H \cong\widetilde{p_{1}}^{*}\widetilde{p_{1}}_{*}\widetilde{q_{1}}^{*} H$ and $\widetilde{p_{1}}^{*}\widetilde{p_{1}}_{*}\xrightarrow{\sim} \operatorname{id}$, the first arrow in the  short exact sequence above is an isomorphism.  Thus,  $\widetilde{q_{1}}^{*}G^{\bullet}=0$, which implies $G^{\bullet}=0$. Hence $ \widetilde{p}^{*}\widetilde{p}_{*} H \xrightarrow{\sim} H $ and then the functor $\widetilde{p}^{*}$ an equivalence. By Lemma \ref{pushforward of algebra}, we have $$ \Coh(\mathbf{B}\mu_{n,k}, M_{n,B}(k))\cong\Coh( \operatorname{Spec}k, p_{*}M_{n,B}(k))\cong \Coh(k[x]/(x^{n}-a)).$$
Thus, we have 
$$\Coh(\mathbf{B}\mu_{n,k}, \cA)\cong\Coh(\mathbf{B}\mu_{n,k}, M_{n,B}(k))\cong \Coh(k[x]/(x^{n}-a)). $$
Let $\cB$ be another Azumaya algebra on $\mathbf{B}\mu_{n,k}$. Suppose $\psi([\cB])=[b]$. Then we have $$\Coh(\mathbf{B}\mu_{n,k}, \cB)\cong \Coh(k[x]/(x^{n}-b)).$$
By Theorem \ref{reconsctruction theorem}, and Corollary \ref{importan corollary}, we know  
$$ \cA \  \text{is Morita equivalent to} \ \cB \Longleftrightarrow k[x]/(x^{n}-a)\cong k[x]/(x^{n}-b) \Longleftrightarrow k(\sqrt[n]{a})=k(\sqrt[n]{b}).$$
So the theorem follows from Proposition \ref{Kummer theory}.

\end{proof}
As a corollary, we can show that Căldăraru's Conjecture holds in this case.

\begin{corollary}\label{hold calda}
Assume $\Br(k)=0$. Then   Căldăraru's Conjecture holds for $\mathbf{B}\mu_{n,k}$. Namely, two Azumaya algebras $\cA$ and $\cB$ on $\mathbf{B}\mu_{n,k}$ are Morita equivalent if and only if there exists an automorphism $\varphi: \mathbf{B}\mu_{n,k}\to \mathbf{B}\mu_{n,k} $ such that $[\cA]=[\varphi^{*}\cB]$ in $\Br(\mathbf{B}\mu_{n,k})$.
\end{corollary}

\begin{proof}
    One direction is clear. For the other direction, note that any automorphism $\varphi_{i}: \mu_{n}\to \mu_{n}: \zeta \to \zeta^{i}  $, where $(i,n)=1$,  induces an automorphism $\varphi_{i}: \mathbf{B}\mu_{n,k}\to \mathbf{B}\mu_{n,k}$. Then for any Azumaya algebra $M_{n,B}(k)$, $\varphi^{*}(M_{n,B}(k))=M_{n, B^{i} }(k)$. Thus $$ \varphi_{i}^{*}: \Br(\mathbf{B}\mu_{n,k})=k^{*}/k^{*n} \to \Br(\mathbf{B}\mu_{n,k})=k^{*}/k^{*n},  [a]\to [a^{i}].$$
    Then the corollary follows from Theorem \ref{Theorem for field}.
\end{proof}

By theorem above, we can construct two Azumaya algebras $\cA, \cB$ on  $\mathbf{B}\mu_{n,k}$ that are Morita equivalent, but $[\cA]\not =[\cB]$ in $\Br(\mathbf{B}\mu_{n,k})$.
\begin{example}\label{Example I}
    Let $k:=\CC(x)$ and $n=4$. Then by Tsen's theorem, $\Br(k)=0$. Let $B, B_{1}$ be the following matrices:

$$B=\begin{pmatrix}
0 & 0 & 0 &  x\\
1 & 0 & 0 &   0 \\
0 & 1 & 0 &  0 \\
 0 & 0 & 1 &  0
\end{pmatrix}, \quad 
B_{1}=\begin{pmatrix}
0 & 0 & 0 &  x^{3}\\
1 & 0 & 0 &   0 \\
0 & 1 & 0 &  0 \\
0 & 0 & 1 &  0
\end{pmatrix},
$$    
then $B^{4}=x\operatorname{Id}$ and $B_{1}^{4}=x^{3}\operatorname{Id}$. Note  $x$ and $x^{3}$ are not the same in the group $\CC(x)^{*}/\CC(x)^{*4}$ but generate the same group of it. By Proposition \ref{important propersition} and Theorem \ref{Theorem for field}, we know $[M_{4,B}(k)]\not =[M_{4,B_{1}}(k)]$ in $\Br(\mathbf{B}\mu_{4, k} )$, but $M_{4,B}(k)$ is Morita equivalent to $M_{4,B_{1}}(k)$.
\end{example}

Now, we consider the general case. To proceed, we need the following twisted version of Gabriel's Theorem.
\begin{lemma}[{\cite[Theorem 1.1]{Ant16}}]\label{twisted theorem antieau}
Let $X$ and $Y$ be Noetherian schemes over $k$, and let $\cA$ and $\cB$ be Azumaya algebras on $X$ and $Y$, respectively. Then 
\begin{center}
$\operatorname{Coh}(X,\cA)\cong \operatorname{Coh}(Y,\cB) $ as $k$-linear abelian categories $  \Longleftrightarrow$ 
there exists an isomorphism $f: X\to Y$ over $k$, such that $[f^{*}\cB]=[\cA]$ in $\Br(X)$.
\end{center} 
\end{lemma}

By Proposition \ref{general non trivial}, we have the following isomorphism
$$\psi: k^{*}/k^{*n}\oplus \Br(k) \xrightarrow{\sim} \Br(\mathbf{B}\mu_{n,k}): ([a], [\cA])\mapsto [M_{n,B}(k)\otimes p^{*}\cA], $$
where $B$ is the matrix described in \ref{matrix B}. Let $X_{a}:=\operatorname{Spec}(k[x]/(x^{n}-a))$ and $q_{a}: X_{a}\to \operatorname{Spec} k$.
Similarly, let \( X_b := \operatorname{Spec}(k[x]/(x^n - b)) \) and \( q_b: X_b \to \operatorname{Spec} k \). We are now ready to state the main theorem.
\begin{theorem}\label{general theorem}
    Let $\cA_{a}$ and $\cB_{b}$ be two Azumaya algebras over $\mathbf{B}\mu_{n,k}$, such that $[\cA_{a}]=\psi([a],[\cA])$ and $[\cB_{b}]=\psi([b],[\cB])$ in $\Br(\mathbf{B}\mu_{n,k})$. Then $\cA_{a}$ is Morita equivalent to $\cB_{b}$ if and only if there exists an isomorphism  f: $X_{a}\to X_{b}$ such that $[q_{a}^{*}\cA]=[f^{*}q^{*}_{b}\cB]$ in $\Br(X_{a})$.
\end{theorem}
\begin{proof}
Let $k_{1}:=k(\sqrt[n]{a})$ and $k_{2}$ be a finite field extension of $k_{1}$ such that $[\cA\otimes k_{2}]=0$ in $\Br(k_{2})$.  We have the following Cartesian diagrams:

    \begin{center}

\begin{tikzcd}
\mathbf{B}\mu_{n,k_{2}} \arrow[r,"q_{2}"] \arrow[d, "p_{2}"] & \mathbf{B}\mu_{n,k_{1}} \arrow[r, "q_{1}"] \arrow[d, "p_{1}"] & \mathbf{B}\mu_{n,k} \arrow[d, "p"] \\
\operatorname{Spec}k_{2} \arrow[r, "h"]                & \operatorname{Spec}k_{1} \arrow[r, "q"]                & \operatorname{Spec}k.              
\end{tikzcd}
 \end{center} 

Let $\cA_{1}:= M_{n,B}(k)\otimes p^{*}\cA $ be the Azumaya algebra over $\mathbf{B}\mu_{n,k}$, where $B$ is the matrix described in \ref{matrix B}. Then $[\cA_{a}]=[\cA_{1}]$ in $\Br(\mathbf{B}\mu_{n,k})$.  By the proof in the Theorem \ref{Theorem for field},  $q_{1}^{*}\cA_{1}\cong \mathcal{E}nd(\Xi)\otimes p_{1}^{*}q^{*}\cA $, where $\Xi:=\rho_{0}\oplus ...\oplus \rho_{n-1}$. By the choice of $k_{2}$, $q_{2}^{*}q_{1}^{*}q^{*}\cA\cong \mathcal{E}nd(E)$ for some vector bundle $E$ on $\mathbf{B}\mu_{n,k_{2}}$. Thus, we have $$q_{2}^{*}q_{1}^{*}\cA_{1}\cong \mathcal{E}nd(\Xi)\otimes \mathcal{E}nd(E)\cong \mathcal{E}nd(\Xi\otimes E).  $$
By Lemma \ref{equivalence lemma},  the functor $$\widetilde{p_{2}}^{*}: \Coh (\operatorname{Spec}k_{2},   p_{2*} q_{2}^{*}q_{1}^{*}\cA_{1} ) \to \Coh(\mathbf{B}\mu_{n,k_{2}},q_{2}^{*}q_{1}^{*}\cA_{1}) $$
is an equivalence. By the same argument as in the proof of Theorem \ref{Theorem for field}, the functor $$\widetilde{p}^{*}:
\Coh (\operatorname{Spec}k,   p_{*} \cA_{1} ) \to \Coh(\mathbf{B}\mu_{n,k},\cA_{1})
$$
is an equivalence. Note that $$p_{*}\cA_{1}=p_{*}(M_{n,B}(k)\otimes p^{*}\cA  )\cong p_{*}M_{n,B}(K)\otimes \cA\cong k[x]/(x^{n}-a)\otimes \cA. $$
So  we have $$\Coh(\mathbf{B}\mu_{n,k},\cA_{a})\cong\Coh(\mathbf{B}\mu_{n,k},\cA_{1}) \cong \Coh(\operatorname{Spec} k, p_{*}\cA_{1})\cong   \Coh(X_{a}, q_{a}^{*}\cA ) . $$
Then the theorem follows from Lemma \ref{twisted theorem antieau}.

\end{proof}

\begin{remark}
     Theorem \ref{Theorem for field} does not hold in general cases. There exists two Azumaya algebras on $\mathbf{B}\mu_{n,k}$ that are Morita equivalent, but do not generate the same subgroup in $\Br(\mathbf{B}\mu_{n,k})$.

\end{remark}
\begin{example}\label{example mortia}
   Recall that $\Br(\RR)=\ZZ/2\ZZ=\langle \RR, \mathbb{H} \rangle$, where $\mathbb{H}$ is the quaternion algebra. By Proposition \ref{general non trivial}, we have
    $$\Br(\mathbf{B}\mu_{2,\RR})= \RR^{*}/\RR^{*2} \oplus \Br(\RR)=  \ZZ/2\ZZ \oplus \ZZ/2\ZZ. $$
Let $B$ be the matrix $$ B:=\begin{pmatrix}
0 & -1 \\
1 & 0 
\end{pmatrix}.$$
Let $\cA:= M_{2,B}(\RR) $ and $\cB:=M_{2,B}(\RR)\otimes p^{*}\mathbb{H}$ be two Azumaya algebras on $\mathbf{B}\mu_{2,\RR}$. Note that $p_{*}\cA=\RR[x]/(x^{2}+1)\cong \CC$ and $p_{*}\cB=\RR[x]/(x^{2}+1)\otimes \mathbb{H}\cong \CC\otimes \mathbb{H}\cong M_{2}(\CC)$. Hence, by the proof of Theorem \ref{general theorem}, we have $$\Coh(\mathbf{B}\mu_{2,\mathbb{R}}, \cA)\cong \Coh(\mathbb{C}), \ \ \text{and} \ \  \Coh(\mathbf{B}\mu_{2,\mathbb{R}}, \cB)\cong \Coh(M_{2}(\mathbb{C}))\cong \Coh (\mathbb{C}). $$
Therefore, $\cA$ is Morita equivalent to $\cB$. However, $[\cA]= \langle \bar{1}, 0\rangle $ and $[\cB]=\langle \bar{1}, \bar{1} \rangle $ in $\Br(\mathbf{B}\mu_{2,\RR})=\ZZ/2\ZZ\oplus \ZZ/2\ZZ$. So $[\cA]$ and $[\cB]$ do not generate the same group.
\end{example}

As a corollary, we can show that, in general,  Căldăraru's conjecture
\textbf{does not} hold for stacks.

\begin{corollary}\label{caldararu}
 Căldăraru's conjecture
does not hold for stacks. More precisely, let $\cA$ and $\cB$ be the two Azumaya algebras from Example \ref{example mortia}, which are Morita equivalent. However, there does not exist an automorphism $\varphi: \mathbf{B}\mu_{2,\mathbb{R}}\to \mathbf{B}\mu_{2,\mathbb{R}} $ over  $\mathbb{R}$ such that $[\varphi^{*}\cB]=[\cA]$ in $\Br(\mathbf{B}\mu_{2,\mathbb{R}})$.
    
\end{corollary}

\begin{proof}
    Recall that $p: \mathbf{B}\mu_{2, \mathbb{R}}\to \operatorname{Spec}(\mathbb{R}) $ is the coarse moduli space. Let $\varphi:\mathbf{B}\mu_{2,\mathbb{R}}\to \mathbf{B}\mu_{2,\mathbb{R}}$ be an automorphism over $\mathbb{R}$. Then $\varphi^{*}\circ p^{*}=p^{*}$. As a result, $$\varphi^{*}: \Br(\mathbf{B}\mu_{2,\mathbb{R}})= \mathbb{R}^{*}/\mathbb{R}^{*2} \oplus \Br(\mathbb{R})\to \mathbb{R}^{*}/\mathbb{R}^{*2}\oplus \Br(\mathbb{R})$$ acts as the identity on 
 $\Br(\mathbb{R})$. Thus, $\varphi^{*}([\cA])=\varphi^{*}(\langle \bar{1},0\rangle) = \langle \bar{1}, 0\rangle \not = \langle \bar{1}, \bar{1}  \rangle=[\cB].$
\end{proof}

By Example \ref{example mortia}, we know if $\cA$ and $\cB$ are Morita equivalent, they may not generate the same subgroup. However, it turns out they must have the same order in the Brauer group.

\begin{corollary}\label{the last corollary}
  Let $\cA$ and $\cB$ be two Azumaya algebras over $\mathbf{B}\mu_{n,k}$. If $\cA$ is Morita equivalent to $\cB$, then $[\cA]$ and $[\cB]$ have the same order in $\Br(\mathbf{B}\mu_{n,k}).$  
\end{corollary}
\begin{proof}
Suppose $[\cA]=\psi([a], [\cA_{1}] )$ and $[\cB]=\psi ([b], [\cB_{1}])$. Then the order of $[\cA]$, $\operatorname{ord}([\cA])$ is 
$$\operatorname{ord}([\cA])=\operatorname{lcm}(\operatorname{ord}([a]), \operatorname{ord}([\cA_{1}]) )=\operatorname{lcm} ([k(\sqrt[n]{a}):k], \operatorname{ord}([\cA_{1}]) ).$$
Since $\cA$ is Morita equivalent to $\cB$, by Theorem \ref{general theorem}, $k(\sqrt[n]{a})=k(\sqrt[n]{b})$ and $ [\cA\otimes k(\sqrt[n]{a})]=[\cB\otimes k(\sqrt[n]{a})] $ in $\Br(k(\sqrt[n]{a}))$. Suppose $[k(\sqrt[n]{a}):k]=d$. Note that we have the restriction map $\operatorname{res}_{k(\sqrt[n]{a})/k}:\Br(k)\to \Br(k(\sqrt[n]{a}))$ and the corestriction map $\operatorname{cores}_{k(\sqrt[n]{a})/k}: \Br(k(\sqrt[n]{a}))\to \Br(k)$. The composition $$\operatorname{cores}_{k(\sqrt[n]{a})/k}\circ \operatorname{res}_{k(\sqrt[n]{a})/k}: \Br(k)\to \Br(k(\sqrt[n]{a}))\to \Br(k) $$
is the multiplication by degree $d$. So we have $d([\cA_{1}]-[\cB_{1}])=\operatorname{cores}_{k(\sqrt[n]{a})/k}\circ \operatorname{res}_{k(\sqrt[n]{a})/k}([\cA_{1}]-[\cB_{1}])=0$. Thus, $d[\cA_{1}]=d[\cB_{1}]$. Since $\operatorname{ord}(d[\cA_{1}])=\frac{\operatorname{ord}([\cA_{1}])}{\operatorname{gcd}(d, \operatorname{ord}([\cA_{1}]))}$, we have

$$\operatorname{ord}([\cA])=\operatorname{lcm}(d, \operatorname{ord}[\cA_{1}])=\frac{d\operatorname{ord}([\cA_{1}])}{\operatorname{gcd}(d, \operatorname{ord}([\cA_{1}]))}=\frac{d\operatorname{ord}([\cB_{1}])}{\operatorname{gcd}(d, \operatorname{ord}([\cB_{1}]))}=\operatorname{ord}([\cB]).$$
We complete the proof.
\end{proof}

\section{globalizing the constructions}\label{section 5}

In this section, we extend the constructions from Section \ref{section 4} to $\mu_{n}$-gerbes over varieties, and prove Theorem \ref{main theorem 2}. Recall that $p: \sX \to X $ is the $\mu_{n}$-gerbe of the line bundle $L$ in Definition \ref{rootgerbe}. By Proposition \ref{exact sequence}, we have an isomorphism   $\psi: \Br(X)\oplus \HH^{1}(X,\mu_{n})\cong \Br(\sX).$ 
Taking cohomology of the short exact sequence (\ref{mu exact sequence}) on $X$, we have
$$
0 \too \Gamma(X,\mathcal{O}_{X}^{*})/\Gamma(X,\mathcal{O}_{X}^{*})^{n} \too \HH^{1}(X,\mu_{n}) \too \Pic(X)[n]\to 0, $$
where $\Pic(X)[n]$ is the group of $n$-torsion line bundles on $X$.

The elements of the group $\HH^{1}(X,\mu_{n})$ can be written as a pair $(\mathcal{L},\alpha)$, where $\mathcal{L}\in \Pic(X)[n] $ and $\alpha$ is a trivialization of $n$-th power of $\mathcal{L}$ \cite[03PK]{stackproject}. The $\mu_{n}$-torsor corresponding to $(\mathcal{L},\alpha)$ is $\widetilde{X}=Spec B \to X $, where $B$ is the algebra $B=\bigoplus_{i=0}^{n-1}\mathcal{L}^{\otimes i} $. The multiplication is given by the natural isomorphism $\mathcal{L}^{\otimes i}\otimes \mathcal{L}^{\otimes j} \cong \mathcal{L}^{\otimes i+j}$ when $i+j<n$, and $$\mathcal{L}^{\otimes i} \otimes \mathcal{L}^{\otimes j} \xrightarrow{\sim} \mathcal{L}^{\otimes(i+j)} \xrightarrow{\alpha \otimes \operatorname{id}}\mathcal{L}^{\otimes (i+j-n)}$$ when $i+j\geq n$.

\subsection{Explicit description of $\HH^{1}(X,\mu_{n})\hookrightarrow \Br(\sX)$}
In this subsection, we explicitly describe the map $i: \HH^{1}(X,\mu_{n})\hookrightarrow \Br(\sX)$
in Proposition \ref{exact sequence}. 

\begin{construction}\label{construction1}
First, we associate to each class $(\mathcal{L}, \alpha) \in \HH^{1}(X,\mu_{n})$ an Azumaya algebra on $\sX$ as follows. Let $\{U_{i}\}$ be an affine open cover of $X$, such that $L|_{U_{i}}\cong \mathcal{O}_{U_{i}}$ for each $U_{i}$, where $L$ is the line bundle in the Definition \ref{rootgerbe}. For each $U_{i}$, we have the following maps 
 $$U_{i}\xrightarrow{\pi_{i}} \mathbf{B}\mu_{n,U_{i}}\xrightarrow{p_{i}} U_{i}.  $$

 Let $F$ be the vector bundle $F:=\bigoplus_{j=0}^{n-1}\mathcal{L}^{\otimes j}$ on $X$. Then there is an isomorphism $F\cong F\otimes \mathcal{L}^{-1}$ induced by $\alpha$. We  will also denote this isomorphism by  $\alpha$. Let $\phi\in \cEnd(F)$ be a local section of $\cEnd(F)$. We define a $\mu_{n}$-action on $\cEnd(F)$ by:
$$\zeta\cdot \phi:= F \xrightarrow{\alpha} F\otimes \mathcal{L}^{-1}\xrightarrow{\phi\otimes \operatorname{id}} F\otimes \mathcal{L}^{-1} \xrightarrow{\alpha^{-1}} F.$$
 
 It turns out that $\cEnd(F)|_{U_{i}}$ are $\mu_{n}$-equivariant algebras for all $U_i$. Thus, we get Azumaya algebras $\cA_{i}$ on $\textbf{B}\mu_{n,U_{i}}$ for all $i$. The readers may check that $\cA_{i}$ can be glued. Therefore we obtain an Azumaya algebra on $\sX$, denote by $\cA_{(\mathcal{L},\alpha)}$.
 \end{construction}
Similar to Lemma \ref{pushforward of algebra}, we have the following generalized version.

\begin{lemma}\label{pushforward of sheaf}
Let $\cA_{(\mathcal{L},\alpha)}$ be the Azumaya algebra associated with the class $(\mathcal{L}, \alpha)$ in Construction \ref{construction1}.
Then we have $p_{*}\cA_{(\mathcal{L},\alpha)}\cong q_{*}\mathcal{O}_{\widetilde{X}}$, where $q: \widetilde{X}\to X$ is the $\mu_{n}$-torsor corresponding to the class $(\mathcal{L},\alpha)\in \HH^{1}(X,\mu_{n})$.
 \end{lemma}

 \begin{proof}
     Choose an affine open covering $X=\bigcup U_{i}$, such that $\mathcal{L}|_{V_{i}}\cong \mathcal{O}_{U_{i}}$ and $L|_{V_{i}}\cong \mathcal{O}_{U_{i}}$. Let $s_{i}\in \mathcal{L}(U_{i})$ be a generator and $\alpha(s_{i}^{\otimes n}): =a_{i}\in \mathcal{O}_{X}(U_{i})^{*}$. Suppose $U_{i}=\operatorname{Spec}(R_{i})$, then $\widetilde{X}|_{U_{i}}=\operatorname{Spec}R_{i}[x]/(x^{n}-a_{i})$.
Note that $\cA_{i}:=\cA_{(\mathcal{L}, \alpha)}|_{\mathbf{B}\mu_{n, U_{i}}}$ is the  Azumaya algebra associated to the $\mu_{n}$-equivariant algebra $\cEnd(F)|_{U_{i}}$. By  Construction \ref{construction1}, we have $\cA_{i}\cong M_{n,B_{i}}(R_{i})$, where $M_{n, B_{i}}(R_{i})$ is the $\mu_{n}$-equivariant algebra over $R_{i}$ described in \ref{matrix notation} and $B_{i}$ is the matrix 
$$B_{i}=\begin{pmatrix}
0 & ... & 0 & 0 & a_{i}\\
1 & ... & 0 & 0 & 0 \\
... & ... & ... & ... & ... \\
0 & ... & 1 & 0 & 0 \\
0 & ... & 0 & 1 & 0
\end{pmatrix}
.$$
By the proof of Lemma \ref{pushforward of algebra}, we have $p_{*}\cA_{i}\cong q_{*}\mathcal{O}_{\widetilde{X}}|_{U_{i}} \cong R_{i}[x]/(x^{n}-a_{i})$. So we have $p_{*}\cA\cong q_{*}\mathcal{O}_{\widetilde{X}}$.

 \end{proof}

It turns out that the embedding $i: \HH^{1}(X,\mu_{n}) \hookrightarrow \Br(\sX)$ is given by Construction \ref{construction1}, as shown in the following proposition.
\begin{proposition}\label{description of injective}
    Explicitly, the map $i: \HH^{1}(X,\mu_{n})\hookrightarrow \Br(\sX) $ in Proposition \ref{exact sequence} is given by $i: (\mathcal{L}, \alpha)\to [\cA_{(\mathcal{L},\alpha)}]$, where $\cA_{(\mathcal{L},\alpha)}$ is in Construction \ref{construction1}.
\end{proposition}
\begin{proof}
 By Lemmas \ref{injective of H} and \ref{HH map},   we have the following commutative diagram.
\begin{center}
\begin{tikzcd}
\HH^{1}(X,\mu_{n}) \arrow[r, "i", hook] \arrow[d, "\eta^{*}", hook] & \Br(\sX) \arrow[d, hook] \\
\HH^{1}(\operatorname{Spec}K(X), \mu_{n}) \arrow[r, "\psi", hook]                      & \Br(\mathbf{B}\mu_{n, K(X)})                     
\end{tikzcd}
\end{center}
Suppose $\eta_{1}^{*}(\mathcal{L},\alpha)=[a]$. By Proposition \ref{general non trivial}, $\psi^{*}([a])=[M_{n,B}(K(X))]$, where $B$ is the $n\times n$-matrix described in \ref{matrix B}. By the description above, $\eta_{3}^{*}( [\cA_{(\mathcal{L},\alpha)}] )=[M_{n,B}(K(X))].$ So $i(\mathcal{L}, \alpha)=[\cA_{(\mathcal{L},\alpha)}].$

\end{proof}
In order to show the main theorem, we need the following lemma.
\begin{lemma}\label{halfmaintheorem}
The functor  $\widetilde{p}^{*}:\Coh(X, p_{*}\cA_{(\mathcal{L},\alpha)})\to \Coh(\sX, \cA_{(\mathcal{L},\alpha)})$ is  an  equivalence.
\end{lemma}
\begin{proof}
    We have following Cartesian diagram:
\begin{center}
\begin{tikzcd}
\widetilde{\sX} \arrow[r, "q_{1}"] \arrow[d, "p_{1}"] & \sX \arrow[d, "p"] \\
\widetilde{X} \arrow[r, "q"]                & X,               
\end{tikzcd}
\end{center}
where $q: \widetilde{X}\to X$ is the $\mu_{n}$-torsor corresponding to the class $(\mathcal{L},\alpha)\in \HH^{1}(X,\mu_{n})$ and $\widetilde{\sX}$ is the $\mu_{n}$-gerbe of line bundle $q^{*}L$. Note that $[q_{1}^{*}\cA_{(\mathcal{L},\alpha)}]=0$ in $\Br(\widetilde{\sX}).$  Then by the same techniques in Theorem \ref{Theorem for field}, we know the functor $\widetilde{p}^{*}$ is an equivalence.
\end{proof}
Now we begin to prove the main theorems in this section. 
We first assume $\Br(X)=0$. In this case, we have $\Br(\sX)=\HH^{1}(X, \mu_{n})$. Let $\cA$ and $\cB$ be Azumaya algebras over $\sX$ such that $[\cA]$ and $[\cB]$ correspond to the $\mu_{n}$-torsors $q_{1}: \widetilde{X}_{1} \to X$ and $q_{2}: \widetilde{X}_{2} \to X$, respectively. Then we have the following theorem, which generalizes Theorem \ref{Theorem for field}.

\begin{theorem}\label{Theorem root gerbe}
Assume $\Br(X)=0$.
	Then $\cA$ and $\cB$ are Morita equivalent if and only if there exists an isomorphism $f: \widetilde{X}_{1}\xrightarrow{\sim} \widetilde{X}_{2}$ over $k$.
\end{theorem}

\begin{proof} 
By Proposition of \ref{description of injective}, there exist $(\mathcal{L}, \alpha )\in \HH^{1}(X, \mu_{n})$,  such that $[\cA]=[\cA_{(\mathcal{L}, \alpha)} ]$. By Lemma \ref{pushforward of sheaf}, we have $$\Coh(\sX, \cA)\cong \Coh(\sX, \cA_{(\mathcal{L},\alpha)})\cong \Coh(X,p_{*}\cA_{(\mathcal{L},\alpha)})\cong \Coh(\widetilde{X}_{1}).$$
The theorem follows from Gabriel's theorem \ref{reconsctruction theorem}.	 
\end{proof}

\begin{remark}
Note that $[\cA]$ and $[\cB]$ generate the same subgroup of $\Br(\sX) = \HH^{1}(X, \mu_{n})$ if and only if there exists an isomorphism $f: \widetilde{X}_{1} \to \widetilde{X}_{2}$ over $X$. In this case, if $[\cA]$ and $[\cB]$ generate the same subgroup of $\Br(\sX)$, then  Theorem \ref{Theorem root gerbe} implies that they are Morita equivalent. However, unlike Theorem \ref{Theorem for field},  Morita equivalent Azumaya algebras on $\sX$ do not necessarily generate the same subgroup of $\Br(\sX)$.
  
\end{remark}
  In general, by Proposition  \ref{exact sequence} we have an isomorphism 
$$\psi: \HH^{1}(X,\mu_{n})\oplus \Br(X)\xrightarrow{\sim} \Br(\sX): ((\mathcal{L}, \alpha), [\cA'])\mapsto [\cA_{(\mathcal{L},\alpha)}\otimes p^{*}\cA'].  $$
 We can generalize Lemma \ref{halfmaintheorem} in the following way.
\begin{lemma}\label{actualmaintheorem}
    We have an equivalence of categories
    $$\widetilde{p}^{*}: \Coh(\tX, q^{*}\cA')\xrightarrow{\sim} \Coh(\sX, \cA_{(\mathcal{L},\alpha)}\otimes p^{*}\cA'), $$
where $q: \widetilde{X}\to X$ is the $\mu_{n}$-torsor corresponding to $(\mathcal{L}, \alpha)\in \HH^{1}(X,\mu_{n})$.
\end{lemma}
\begin{proof}
    This follows from the proof of Theorem \ref{general theorem}.
\end{proof}

Note that Lemma \ref{actualmaintheorem} has many interesting applications. For instance, it demonstrates that a decomposable category can become indecomposable under a Brauer twist. One such example is as follows:
\begin{example}\label{elliptic}
Let $X$ be an elliptic curve over $\CC$, and $\sX$ be the $\mu_{n}$-gerbe of any line bundle. Then we have $$D^{b}(\sX)\cong \oplus_{k=0}^{n-1} D^{b}(X)\rho_{k}. $$ 
Let $\cA'=0$, $0\neq \cL\in \HH^{1}(X, \mu_{2})=\Pic(X)[2]$, and $\cA_{(\cL, 1)}$ be the Azumaya algebra over $\sX$ defined in Construction \ref{construction1}. Then by Lemma \ref{actualmaintheorem}, we have 
$$D^{b}(\sX, \cA_{(\cL, 1)}) \cong D^{b}(\tX), $$
where $\tX$ is also an elliptic curve. Hence $D^{b}(\sX, \cA_{(\cL, 1)})$ is indecomposable.
\end{example}

 Let $\cA$ and $\cB$ be two Azumaya algebras over $\sX$, such that $[\cA]=\psi((\mathcal{L}_{1}, \alpha_{1}), [\cA'])$ and $[\cB]=\psi((\mathcal{L}_{2}, \alpha_{2} ), [\cB'])$ in $\Br(\sX)$. Let $q_{1}: \widetilde{X}_{1}\to X, \ q_{2}: \widetilde{X}_{2}\to X$ be the corresponding $\mu_{n}$-torsors defined by $(\mathcal{L}_{1}, \alpha_{1})$ and $(\mathcal{L}_{2}, \alpha_{2})$, respectively. As Theorem \ref{general theorem}, we have the following theorem.
 \begin{theorem}\label{Theorem root gerbe general}     Let $\cA$ and $\cB$ be two Azumaya algebras as above. Then $\cA$ and $\cB$ are Morita equivalent if and only if   there exists an isomorphism $f: \widetilde{X}_{1}\xrightarrow{\sim} \widetilde{X}_{2}$  over $k$, such that $[q_{1}^{*}\cA']=[f^{*}q_{2}^{*}\cB']$ in $\Br(\widetilde{X}_{1})$.
 \end{theorem}
 \begin{proof}
     This theorem follow from Lemma \ref{twisted theorem antieau} and Lemma \ref{actualmaintheorem}.
 \end{proof}

\begin{remark}
It is reasonable to expect that the theorems above can be generalized to $\mu_{n}$-gerbes of line bundles over Noetherian schemes using the same idea.
\end{remark}

\section*{References}
\bibliographystyle{alpha}
\renewcommand{\section}[2]{} 
\bibliography{ref}

\end{document}